\documentclass[11pt,reqno]{amsart}
\usepackage[usenames]{color}
\usepackage{amsmath, amssymb, latexsym, multicol, amsthm, color, enumitem, tensor, comment, mathrsfs, tikz}
\usepackage[all]{xy}
\usepackage{cancel}
\usepackage[T1]{fontenc}
\usepackage{hyperref}
\usepackage{cleveref}
\newtheoremstyle{natural}
{\parsep}   
{\parsep}   
{\normalfont}  
{0pt}       
{\bfseries} 
{.}         
{5pt plus 1pt minus 1pt} 
{}          
\newtheoremstyle{remark}
{\parsep}   
{\parsep}   
{\normalfont}  
{0pt}       
{\bfseries} 
{.}         
{5pt plus 1pt minus 1pt} 
{}          
\theoremstyle{plain}
\newtheorem{thm}{Theorem}[section]
\newtheorem{cor}[thm]{Corollary}

\newtheorem{prp}[thm]{Proposition}
\newtheorem{cnj}[thm]{Conjecture}

\theoremstyle{remark}
\newtheorem{expl}[thm]{Example}
\newtheorem*{rmk}{Remark}

\numberwithin{equation}{section}

\crefname{thm}{theorem}{theorems}
\crefname{lem}{lemma}{lemmata}
\crefname{prp}{proposition}{propositions}
\crefname{cor}{corollary}{corollaries}

\newcommand{\N}{\mathbb{N}}

\newcommand{\Q}{\mathbb{Q}}
\newcommand{\R}{\mathbb{R}}
\newcommand{\Z}{\mathbb{Z}}

\newcommand{\mc}{\mathcal}

\newcommand{\cb}[1]{\left\{{#1}\right\}}
\newcommand{\cbm}[2]{\left\{{#1}\; \middle|\; {#2}\right\}}

\newcommand{\pb}[1]{\left\langle{#1}\right\rangle}
\newcommand{\sqb}[1]{\left[{#1}\right]}
\newcommand{\rb}[1]{\left({#1}\right)}
\newcommand{\vb}[1]{\left| {#1} \right|}

\newcommand{\bB}{\begin{Bmatrix}}
\newcommand{\eB}{\end{Bmatrix}}
\newcommand{\bb}{\begin{bmatrix}}
\newcommand{\eb}{\end{bmatrix}}
\newcommand{\bp}{\begin{pmatrix}}
\newcommand{\ep}{\end{pmatrix}}
\newcommand{\bsm}{\left(\begin{smallmatrix}}
\newcommand{\esm}{\end{smallmatrix}\right)}
\newcommand{\bv}{\begin{vmatrix}}
\newcommand{\ev}{\end{vmatrix}}

\newcommand{\fa}{\;\forall}

\newcommand{\Gal}{\operatorname{Gal}}

\newcommand{\hra}{\hookrightarrow}

\newcommand{\NN}{\operatorname{N}}

\newcommand{\ol}[1]{\overline{#1}}

\newcommand{\rad}{\operatorname{rad}}

\newcommand{\sbe}{\subseteq}

\newcommand{\sgn}{\operatorname{sgn}}

\newcommand{\Tr}{\operatorname{Tr}}

\newcommand{\ba}{\[\begin{aligned}} 
\newcommand{\ea}{\end{aligned}\]} 
\newcommand{\bd}{\begin{tikzcd}} 
\newcommand{\ed}{\end{tikzcd}} 
\newcommand{\bdd}{\begin{center}\begin{tikzcd}} 
\newcommand{\edd}{\end{tikzcd}\end{center}} 
\newcommand{\bdp}{\begin{center}\begin{tikzpicture}} 
\newcommand{\edp}{\end{tikzpicture}\end{center}} 
\newcommand{\bi}{\begin{itemize}} 
\newcommand{\ei}{\end{itemize}} 
\newcommand{\bea}{\begin{enumerate}[label=(\alph*)]} 
\newcommand{\ben}{\begin{enumerate}[label=(\arabic*)]} 
\newcommand{\ber}{\begin{enumerate}[label=(\roman*)]} 
\newcommand{\beani}{\begin{enumerate}[label=(\alph*), wide, labelwidth=!, labelindent=0pt]} 
\newcommand{\benni}{\begin{enumerate}[label=(\arabic*), wide, labelwidth=!, labelindent=0pt]} 
\newcommand{\berni}{\begin{enumerate}[label=(\roman*), wide, labelwidth=!, labelindent=0pt]} 
\newcommand{\ee}{\end{enumerate}} 

\title[Universal lattices and indecomposable elements in real multiquadratic fields]{Minimal rank of universal lattices and number of indecomposable elements in real multiquadratic fields}

\author{Siu Hang Man}
\address{Siu Hang Man, Charles University, Faculty of Mathematics and Physics, Department of Algebra, Sokolovská 49/83, 186 75 Praha 8, Czechia}
\email{shman@karlin.mff.cuni.cz}

\thanks{The author is supported by the Czech Science Foundation GAČR grant 21-00420M, and the OP RDE project No. CZ.02.2.69/0.0/0.0/18\_053/0016976 International mobility of research, technical and administrative staff at the Charles University.}

\begin{document}

\begin{abstract}
We establish an upper bound on the number of real multiquadratic fields that admit a universal quadratic lattice of a given rank, or contain a given amount of indecomposable elements modulo totally positive units, obtaining density zero statements. We also study the structure of indecomposable elements in real biquadratic fields, and compute a system of indecomposable elements modulo totally positive units for some families of real biquadratic fields.
\end{abstract}
\date{\today}
\makeatletter
\@namedef{subjclassname@2020}{%
  \textup{2020} Mathematics Subject Classification}
\makeatother
\subjclass[2020]{11A55, 11E12, 11E20, 11H55, 11R20, 11R80}
\keywords{universal quadratic form, quadratic lattice, indecomposable elements}
\maketitle


\section{Introduction}

The question of which integers can be represented by a given quadratic form is a long-standing topic in number theory, which fascinated numerous mathematicians since ancient times. One particularly interesting question is that whether a positive definite quadratic form is \emph{universal}, i.e. represents all positive integers. The first important result in this direction is Lagrange's four-squares theorem, which says the quadratic form $x^2+y^2+z^2+w^2$ is universal. Much work has been done since then, including the renowned $15$- and $290$-theorems of Conway-Schneeberger \cite{Bhargava2000} and Bhargava-Hanke \cite{BHp2011}. 

In a more general setting, when we are given a totally real number field $F$ and their ring of integers $\mc O_F$, we may ask whether a totally positive definite quadratic form is universal, i.e. represents all totally positive elements in $\mc O_F$. As a starting point, Maaß \cite{Maass1941} showed that the sum of three squares is universal over the ring of integers of $\Q(\sqrt{5})$, using theta series. Conversely, Siegel \cite{Siegel1945} proved that the sum of any number of squares is universal only over $F = \Q, \Q(\sqrt{5})$. Siegel's proof made heavy use of the notion of \emph{indecomposable} elements (under the name ``extremal elements''), i.e. totally positive algebraic integers in $F$ that cannot be written as the sum of two other totally positive integers in $F$. 

For a more precise description, we consider totally positive definite \emph{quadratic $\mc O_F$-lattices} $(\Lambda,Q)$ over a totally real number field $F$, that is, finitely generated $\mc O_F$-modules $\Lambda$ equipped with a quadratic form $Q$ such that all the values $Q(v) \ne 0$ for $0\ne v\in\Lambda$ are totally positive elements in $\mc O_F$. Such a lattice is called \emph{universal} if it represents all the totally positive integers. A convenient assumption that we will often use is that $(\Lambda,Q)$ is \emph{classical}, that is, all the values in the associated symmetric bilinear form lie in $\mc O_F$. Since it has been shown that universal classical lattices exist over every totally real number field \cite{HKK1978}, we can denote by $R_{\operatorname{cls}}(F)$ the minimal rank of a universal classical lattice over $F$, and $R(F)$ the minimal rank of a universal lattice over $F$ without the classical assumption. Here the rank of a quadratic lattice $(\Lambda,Q)$ over $F$ is defined as the dimension of $F\Lambda$ as an $F$-vector space.

One can then ask about how the minimal ranks $R(F)$ and $R_{\operatorname{cls}}(F)$ behave in a collection of totally real number fields $F$. In general, it is widely expected that the minimal rank of universal lattices is usually high, as illustrated by an influential conjecture of Kitaoka (see \cite{CKR1996}), which claims that there are only finitely many totally real number fields $F$ with a ternary universal lattice (i.e. with $R(F)\le 3$). 

For specific cases more is known. For instance, it is known that there are only finitely many real quadratic fields $F$ with $R(F) \le 7$ \cite{KKP2022}, while there are infinitely many real quadratic fields with $R(F) = R_{\operatorname{cls}}(F) = 8$ \cite{Kim2000}. Nevertheless, it is known that for most real quadratic fields the minimal ranks of universal lattices are large \cite{KYZp2023}. Meanwhile, Kala and Svoboda \cite{KS2019,Kala2023} showed that higher degree fields for which $R(F)$ is large are also quite easy to find.

In many instances, the minimal rank of universal lattices over a totally real number field $F$ is closely related to the number of indecomposable elements up to multiplication of totally positive units (denoted by $\iota(F)$), the rough idea being that every totally positive integer is a sum of indecomposable elements, while indecomposable elements themselves are quite difficult to represent. Using this idea, Kala and Tinková \cite{KT2023} showed that the minimal rank of universal lattices is bounded above in terms of the number of indecomposable elements. In the other direction, by proving the existence of sufficiently many indecomposable elements satisfying certain properties, one can show that in such fields the minimal rank of universal lattices is high; there are numerous literature tackling different cases, see for example \cite{BK2015,BK2018,Kala2016,KYZp2023} about quadratic fields, and \cite{CLSTZ2019,KS2019,KT2023,KTZ2020} about number fields of higher degrees.

In this paper, we generalise the results in \cite{KYZp2023}, and show that for most real multiquadratic fields $K$, the minimal rank $R(K)$ of a universal lattice is large. For $n\in\N$, and $X\ge 1$ a positive parameter, let $\mc K(n,X)$ denote the set of real multiquadratic fields $K$ of degree $2^n$ with discriminant $\Delta_K \le X$. Note that by a result of Wright \cite{Wright1989} we have $\#\mc K(n,X) \asymp X^{2^{1-n}} (\log X)^{2^n-2}$ (see \Cref{thm:Wright}). 

\begin{thm}\label{thm:main1}
Let $n\in\N$ be fixed, and $\varepsilon > 0$. Then for almost all real multiquadratic fields of degree $2^n$ we have
\ba
R_{\operatorname{cls}}(K) \ge \begin{cases} \Delta_K^{\frac 1{72}-\varepsilon} & \text{ for } n=2,\\ \Delta_K^{\frac 1{12(2^n-1)^2}-\varepsilon} & \text{ otherwise,}\end{cases} \quad \text{ and } \quad R(K) \ge \begin{cases} \Delta_K^{\frac 1{24}-\varepsilon} & \text{ for } n=1,\\ \Delta_K^{\frac{1}{12(2^n-1)(2^{n+1}-1)}-\varepsilon} & \text{ otherwise.}\end{cases}
\ea
\end{thm}

\begin{thm}\label{thm:main2}
Let $n\in\N$ be fixed, and $\varepsilon > 0$. Then for almost all real multiquadratic fields of degree $2^n$ we have
\begin{align*}
\iota(K) &\ge \begin{cases} \Delta_K^{\frac 1{72}-\varepsilon} & \text{ for } n=2,\\ \Delta_K^{\frac 1{12(2^n-1)^2}-\varepsilon} & \text{ otherwise.}\end{cases}
\end{align*}
\end{thm}
Here, by ``almost all'' we mean that the set of such fields has natural density $1$, with respect to the ordering of real multiquadratic fields of degree $2^n$ by discriminant.

\Cref{thm:main1,thm:main2} follow from our main results below, whose proofs involve a wide variety of tools. To state our main results, we work on a slightly more general setting, considering quadratic lattices that represent all multiples of a fixed positive integer $m$. We say that a totally positive definite quadratic lattice $(\Lambda,Q)$ over a totally real number field $F$ is \emph{$m\mc O_F$-universal} if $(\Lambda,Q)$ represents all elements of $m\mc O_F^+$, i.e. all the totally positive multiples of $m$. This allows us to convert an arbitrary universal quadratic lattice $(\Lambda,Q)$ into a $2\mc O_F$-universal classical quadratic lattice $(\Lambda,2Q)$, which is convenient because classical lattices are often easier to work with.

Our first main result gives an explicit upper bound on the number of real multiquadratic fields $K$ which admits an $m\mc O_K$-universal classical lattice of a given rank. 

\begin{thm}\label{thm:fuf} 
Let $n,m\in\N$ be fixed. For $R\in\N$ and $X\ge 1$ we define
\ba
\mc K_{\operatorname{univ}}(n,X,R,m) := \cbm{K\in\mc K(n,X)}{\exists \; m\mc O_K \text{-universal quadratic lattice of rank R}}.
\ea
Then, for any $\varepsilon > 0$ and $X \gg (C(2^nR, 2^{n-1}m)\log X)^{\frac{4(2^n-1)^2}{2^n-2^{n-4}-1}}$ for a sufficiently large constant depending only on $n$, we have
\ba
\#\mc K_{\operatorname{univ}}(n,X,R,m) \ll_{n,\varepsilon} C(2^nR, 2^{n-1}m)^{\frac 32} X^{2^{1-n}\rb{1-\frac{2^{n-4}}{2^n-1}}+\varepsilon} + C(2^nR, 2^{n-1}m)^3 X^{2^{1-n}\rb{1-\frac{2^{n-3}}{2^n-1}}+\varepsilon},
\ea
where $C(2^nR, 2^{n-1}m)$ is an explicit constant defined in \eqref{eq:crmd}.
\end{thm}

Here we give the brief strategy to the proof, assuming $m=1$ for simplicity. From \cite{KT2023}, we know that for real quadratic fields $F$, we may find a totally positive element $\delta \in \mc O_F^{\vee,+}$ in the codifferent such that there are $w$ indecomposable elements $\alpha_1,\ldots,\alpha_w \in \mc O_F^+$ satisfying $\Tr_{F/\Q}(\delta\alpha_i) = 1$. In a multiquadratic field $K$ of degree $2^n$ containing $F$, these elements satisfy $\Tr_{K/\Q}(\delta\alpha_i) = 2^{n-1}$. Viewing $\mc O_K$ as a $\Z$-lattice $\Lambda$ equipped with the $\delta$-twisted trace form, we obtain a set of lattice points $v\in\Lambda$ of norm $2^{n-1}$. However, the number of such points in a totally positive definite $\Z$-lattice of rank $R$ is bounded (see \cite[Theorem~3.1]{KYZp2023} and \cite[Theorem~1.1]{RSDp2023}). This gives an upper bound to the size of $w$. Meanwhile, \cite[Corollary 2.13]{KYZp2023} says that $w$ is usually large. Combining these estimates, we are able to show that few real multiquadratic fields admit universal quadratic lattices of rank $R$.

Using the linkage between universal lattices and indecomposable elements in \cite{KT2023}, we obtain as a consequence of \Cref{thm:fuf} that for given $R\in\N$, few multiquadratic fields $K$ satisfies $\iota(K) \le R$.

\begin{thm}\label{thm:fi} 
Let $n\in\N$. For $R\in\N$ and $X\ge 1$ we define
\ba
\mc K_{\operatorname{indec}}(n,X,R) := \cbm{K\in\mc K(n,X)}{\iota(K) \le R}.
\ea
Then, for any $\varepsilon>0$ and $X \gg (C(2^nR, 2^{n-1})\log X)^{\frac{4(2^n-1)^2}{2^n-2^{n-4}-1}}$ for a sufficiently large constant depending only on $n$, we have
\ba
\#\mc K_{\operatorname{indec}}(n,X,R) \ll_{n,\varepsilon} C(2^nR, 2^{n-1})^{\frac 32} X^{2^{1-n}\rb{1-\frac{2^{n-4}}{2^n-1}}+\varepsilon} + C(2^nR, 2^{n-1})^3 X^{2^{1-n}\rb{1-\frac{2^{n-3}}{2^n-1}}+\varepsilon},
\ea
where $C(2^nR, 2^{n-1})$ is an explicit constant defined in \eqref{eq:crmd}.
\end{thm}

The linkage between universal lattices and indecomposable elements also suggests that we study the structure of indecomposable elements in totally real number fields. While this is well understood for real quadratic fields, little is known about the structure of indecomposable elements in other cases. For example, for totally real cubic fields we have an explicit description of the indecomposable elements only for Shank's family of simplest cubic fields \cite{GMTp2022,KT2023}. In the second part of the paper we study the structure of indecomposable elements in biquadratic fields. Since biquadratic fields contain quadratic subfields, it is then natural to ask whether the indecomposable elements in the quadratic subfields, which are well understood, remain indecomposable in the biquadratic field. In general the answer is no, but the following theorem says that the answer is yes for at least two of the three quadratic subfields.

\begin{thm}\label{thm:pii} 
Let $K$ be a real biquadratic field, with quadratic subfields $K_1, K_2, K_3$. Then, for at least two of the quadratic subfields $K_j$, $j\in\cb{1,2,3}$, the indecomposable elements in $\mc O_{K_j}^+$ remain indecomposable in $\mc O_K^+$. 
\end{thm}

\Cref{thm:pii} represents an improvement to the result in \cite{CLSTZ2019} which established the indecomposability of indecomposable elements from quadratic subfields subject to a criterion involving continued fractions, and answers a conjecture in \cite{KTZ2020}. The proof relies on the correspondence between indecomposable elements in real quadratic fields and best one-sided Diophantine approximations. Since it is known that real quadratic fields $F$ with $\iota(F)\le R$ has density zero, \Cref{thm:pii} yields an alternative proof that real biquadratic fields $K$ with $\iota(K)\le R$ has density zero.

In \Cref{sect:if} we compute the structure of indecomposable elements for a few families of real biquadratic fields. The families of biquadratic fields $K$ chosen here are special in the sense that each quadratic subfield $F\sbe K$ has $\iota(F) = 1$. In view of \Cref{thm:pii}, these families represent likely candidates for which $\iota(K)$ is small. The following theorem says that $\iota(K)$ may still grow within such a family.

\begin{thm}\label{thm:f1} 
Let $n\ge 6$ be an integer such that
\ba
p &= (2n-1)(2n+1), & q&=(2n-1)(2n+3), & r&=(2n+1)(2n+3)
\ea
are squarefree integers, and let $K = \Q(\sqrt{p},\sqrt{q})$. A complete set of indecomposable elements in $\mc O_K^+$ modulo totally positive units is given by 
\ba
&1, \quad \frac{\varepsilon_p^{-1}+\varepsilon_r}2, \quad \mu := \rb{n+\frac 32} + \frac 12\sqrt{p}+\frac 12\sqrt{q}+\frac 12\sqrt{r},\\
&1+\varepsilon_p + t(\mu-1), \quad \frac{1+\varepsilon_p\varepsilon_r}2 + \varepsilon_p + t(\mu-1) & &(3\le t \le 2n-2),\\
&1+\varepsilon_q^{-1}+t(\mu-1), \quad \frac{1+\varepsilon_r\varepsilon_p}2 + \varepsilon_q^{-1} + t(\mu-1) & &(4\le t\le 2n-1),\\
&\frac{\varepsilon_r^{-1}+\varepsilon_p}2 + t(\mu-2) & &(2\le t \le 2n-1).
\ea
Here, $\varepsilon_p, \varepsilon_q, \varepsilon_r$ stand for the fundamental units of $\Q(\sqrt{p}), \Q(\sqrt{q})$, $\Q(\sqrt{r})$ respectively, chosen such that $\varepsilon_p,\varepsilon_q,\varepsilon_r > 1$. In particular, we have $\iota(K) = 10n-15$.
\end{thm}

\begin{rmk}
The infinitude of the family in \Cref{thm:f1} is equivalent to that the polynomial $f(n) = (2n-1)(2n+1)(2n+3)$ takes infinitely many squarefree values. Since $f(n)$ is a product of linear terms, such a statement can be conveniently proved using standard sieve arguments, see for example \cite[Section~2]{Erdos1953}.
\end{rmk}

In analogous \Cref{thm:f2,thm:f3}, we compute the set of indecomposable elements in two other one-parameter families of biquadratic fields, and show that $\iota(K)$ also grows within such families. In \Cref{sect:cr} we compute $\iota(K)$ for a number of real biquadratic fields $K$ with small discriminants. The computations suggest that the one-parameter families in \Cref{thm:f2,thm:f3} are among the fields for which $\iota(K)$ is the smallest. In view of this observation, we formulate the following conjecture.

\begin{cnj}\label{cnj:cnj} 
For any $R\in\N$, there are only finitely many real biquadratic fields $K$ for which we have $\iota(K) \le R$.
\end{cnj}

The paper is organised as follows. In \Cref{sect:pre} we recall basic results on multiquadratic fields, as well as the theory of indecomposable elements for real quadratic fields. In \Cref{sect:den} we prove \Cref{thm:main1,thm:main2,thm:fuf,thm:fi}. In \Cref{sect:pi} we prove \Cref{thm:pii} (as \Cref{thm:pid}). Finally, in \Cref{sect:if} we compute the structure of indecomposable elements for a few families of biquadratic fields, proving \Cref{thm:f1} and its analogues.

\subsection*{Notations}
Here we recollect the notations used in the paper which are possibly non-standard. Throughout, $F$ stands for a totally real number field.
\begin{center}
\begin{tabular}{rl}
$\Delta_F$ & the discriminant of $F$\\
$\mc O_F$ (resp. $\mc O_F^+$) & the ring of integers (resp. the set of totally positive integers) of $F$\\
$\mc O_F^\times$ (resp. $\mc O_F^{\times,+}$) & the group of units (resp. totally positive units) in $\mc O_F$\\
$R(F)$ & the minimal rank of universal lattices over $F$\\
$R_{\operatorname{cls}}(F)$ & the minimal rank of classical universal lattices over $F$\\
$\iota(F)$ & the number of indecomposable elements in $\mc O_F^+$ modulo totally positive units\\
$\mc K(n,X)$ & the set of real multiquadratic fields of degree $2^n$ with discriminant $\le X$\\
$\mc K_{\operatorname{univ}}(n,X,R,m)$ & the set of fields $K\in \mc K(n,X)$ admitting $m\mc O_K$-universal classical lattice of rank $R$\\
$\mc K_{\operatorname{indec}}(n,X,R)$ & the set of fields $K\in\mc K(n,X)$ with $\iota(K) \le R$
\end{tabular}
\end{center}

Throughout the article, we use the following asymptotic notations. For real functions $f(x), g(x)$, we write $f\ll g$ or $g\gg f$ if there are $c>0$ and $x_0$ such that $\vb{f(x)} \le c g(x)$ for all $x\ge x_0$; we may write $f\ll_y g$ to emphasise that the constant $c$ depends on $y$. We write $f\asymp g$ if $f\ll g$ and $g\gg f$.

\section*{Acknowledgement}

The author would like to thank Vítězslav Kala and Błażej Żmija for the helpful discussions, and the referee for the useful suggestions that improved the presentation and simplified some proofs.

\section{Preliminaries}\label{sect:pre}

\subsection{Multiquadratic fields}\label{sect:mf} 
Let $F$ be a totally real number field. We denote by $\mc O_F$ the ring of integers of $F$. An element $\alpha \in F$ is called \emph{totally positive} if $\sigma(\alpha) > 0$ for every embedding $\sigma:F\hra\R$. This is denoted by $\alpha\succ 0$. We denote by $\mc O_F^+$ the set of totally positive elements in $\mc O_F$. An element $\alpha\in \mc O_F^+$ is called \emph{indecomposable} if it cannot be written as a sum $\alpha = \beta+\gamma$ with $\beta,\gamma \in \mc O_F^+$. It is well-known that there are finitely many indecomposable elements up to multiplication by totally positive units. For an easy proof of this fact, we note that there exists some constant $c_F$ such that $\NN_{F/\Q}(\alpha) \le c_F$ if $\alpha\in\mc O_F^+$ is indecomposable (see \cite{DS1982}), and that up to multiplication by totally positive units there are only finitely many $\alpha\in\mc O_F^+$ with $\NN_{F/\Q}(\alpha) \le c_F$ (see \cite[I.7.2]{Neukirch1999}). Throughout the article, we shall denote by $\iota(F)$ the number of indecomposable elements up to multiplication by totally positive units. When no confusion may arise, we simply call $\iota(F)$ the number of indecomposable elements of $\mc O_F$.

We let $K = \Q(\sqrt{A_{2^0}}, \sqrt{A_{2^1}}, \ldots, \sqrt{A_{2^{n-1}}})$ be a multiquadratic field of degree $[K:\Q] = 2^n$, where $n\ge 2$, and $A_{2^0}$, \ldots, $A_{2^{n-1}}$ are squarefree integers. By picking different generators of $K$, one can impose some congruence conditions on $A_{2^0},\ldots, A_{2^{n-1}}$, which is helpful in narrowing down the number of cases to consider. It was shown in \cite{Chatelain1973,Schmal1989} that every multiquadratic field $K$ can be written in a way so that we have
\begin{equation}\label{eq:mfa} 
\begin{aligned}
A_{2^k} &\equiv 1 \pmod{4}, \quad 0\le k \le n-3, \text{ and}\\ (A_{2^{n-2}}, A_{2^{n-1}}) &\equiv (1,1),\; (1,2), \; (1,3), \text{ or } (2,3) \pmod{4}.
\end{aligned}
\end{equation}
We will always assume that the integers $A_{2^k}$ are of the form above. 

The field $K$ contains $2^n-1$ quadratic subfields. To describe them, it is helpful to establish some notations. For $0\le j \le 2^n-1$, we define integers $A_j \in \Z$ as follows. We set $A_0 := 1$. Then, if $j = 2^k + j'$ with $0\le j' \le 2^k-1$, we define $A_j := \frac{A_{2^k} A_{j'}}{d_{k,j}^2}$, where $d_{k,j}$ is a greatest common divisor of $A_{2^k}$ and $A_{j'}$. Replacing $d_{k,j}$ by $-d_{k,j}$ when necessary, we may assume $d_{k,j} \equiv 1 \pmod{4}$. We note that the integers $A_j$ are squarefree by construction. The $2^n-1$ quadratic subfields of $K$ are then given by $\Q(\sqrt{A_j})$, for $1\le j \le 2^n-1$. 

The ring of integers $\mc O_K$ of $K$ is a free $\Z$-module, hence admits a $\Z$-basis. Let $\alpha_0 = 1$, and $\alpha_{2^k} = \sqrt{A_{2^k}}$, for $0\le k \le n-1$. If $j = 2^k + j'$, where $1\le j' \le 2^k-1$, we define $\alpha_j = \frac{\alpha_{2^k}\alpha_{j'}}{d_{k,j}}$, where $d_{k,j}$ is defined as above. A $\Z$-basis of $\mc O_K$ is then given in the following theorem of Chatelain \cite{Chatelain1973}.

\begin{thm}[\cite{Chatelain1973}]\label{thm:Chatelain}
Let $K = \Q(\sqrt{A_{2^0}}, \sqrt{A_{2^1}}, \ldots, \sqrt{A_{2^{n-1}}})$ be a multiquadratic field of degree $2^n$, where $n\ge 2$, and $A_{2^0}$, \ldots, $A_{2^{n-1}}$ are squarefree integers satisfying \eqref{eq:mfa}. Assuming the notations above, a $\Z$-basis of $\mc O_K$ is given as follows:
\ben
\item Suppose $(A_{2^{n-2}}, A_{2^{n-1}}) \equiv (1,1) \pmod{4}$. Define $E := 2^{-n} \sum_{j=0}^{2^n-1} \alpha_j$. Then a $\Z$-basis of $\mc O_K$ is given by the set $\cbm{\sigma(E)}{\sigma\in\Gal(K/\Q)}$.
\item Suppose $(A_{2^{n-2}}, A_{2^{n-1}}) \equiv (1,2)$ or $(1,3) \pmod{4}$. Define
\begin{align*}
E_1 &:= 2^{-n+1} \sum_{j=0}^{2^{n-1}-1} \alpha_j, & E_2 &:= 2^{-n+1} \sum_{j=2^{n-1}}^{2^n-1} \alpha_j.
\end{align*}
Then a $\Z$-basis of $\mc O_K$ is given by the set $\cbm{\sigma(E_i)}{\sigma\in \Gal(K/\Q(\sqrt{A_{2^{n-1}}})), \; i=1,2}$.
\item Suppose $(A_{2^{n-2}}, A_{2^{n-1}}) \equiv (2,3) \pmod{4}$. Define
\begin{align*}
E_1 &:= 2^{-n+2} \sum_{j=0}^{2^{n-2}-1} \alpha_j, & E_2 &:= 2^{-n+2} \sum_{j=2^{n-2}}^{2^{n-1}-1} \alpha_j,\\
E_3 &:= 2^{-n+2} \sum_{j=2^{n-1}}^{3\cdot 2^{n-2}-1} \alpha_j, & E_4 &= 2^{-n+1} \rb{\sum_{j=2^{n-2}}^{2^{n-1}-1} \alpha_j + \sum_{j=3\cdot 2^{n-2}}^{2^n-1} \alpha_j}.
\end{align*}
Then a $\Z$-basis of $\mc O_K$ is given by the set $\cbm{\sigma(E_i)}{\sigma\in \Gal(K/\Q(\sqrt{A_{2^{n-2}}},\sqrt{A_{2^{n-1}}})),\; i=1,2,3,4}$.
\ee
\end{thm}

We shall not need the explicit shape of the integral basis of $\mc O_K$ beyond the biquadratic case in this article. But \Cref{thm:Chatelain} gives the following convenient corollary on the discriminant of multiquadratic fields (see \cite{Chatelain1973,Schmal1989}).

\begin{cor}[{\cite[Satz~2.1]{Schmal1989}}]
Assume the settings above. Then the discriminant of the multiquadratic field $K$ is given by
\ba
\Delta_K = \rb{2^r \rad\rb{A_{2^0}A_{2^1}\cdots A_{2^{n-1}}}}^{2^{n-1}} = 2^{2^{n-1}r} \prod_{j=1}^{2^n-1} A_j,
\ea
where $\rad(x)$ denotes the radical of $x\in\Z$, and
\begin{equation}\label{eq:drd} 
r = \begin{cases} 0 & \text{ if } (A_{2^{n-2}},A_{2^{n-1}}) \equiv (1,1)\pmod{4},\\ 2 & \text{ if } (A_{2^{n-2}}, A_{2^{n-1}}) \equiv (1,2) \text{ or } (1,3) \pmod{4},\\ 3 & \text{ if } (A_{2^{n-2}}, A_{2^{n-1}}) \equiv (2,3) \pmod{4}.\end{cases}
\end{equation}
\end{cor}

We shall also need an asymptotic formula on the number of multiquadratic fields of degree $2^n$ with bounded discriminant. For this we recall a result of Wright \cite{Wright1989}, applied to multiquadratic fields.
\begin{thm}[{\cite[Theorem~I.2]{Wright1989}}]\label{thm:Wright}
Let $X\ge 1$, and let $\mc K(n,X)$ denote the set of real multiquadratic fields of degree $2^n$ with discriminant $\Delta_K \le X$. Then we have
\[
\# \mc K(n,X) \asymp X^{2^{1-n}} (\log X)^{2^n-2}.
\]
\end{thm}
\begin{rmk}
Strictly speaking, Wright's result gives an asymptotic formula to the set of multiquadratic fields of degree $2^n$ with discriminant $\Delta_K \le X$, without distinguishing whether they are real. Through the mapping $\Q(\sqrt{A_{2^0}}, \ldots, \sqrt{A_{2^{n-1}}}) \mapsto \Q(\sqrt{|A_{2^0}|}, \ldots, \sqrt{|A_{2^{n-1}}|})$, it is easy to see that the subset of real fields also satisfy the same asymptotic formula. An alternative proof of the statement is also given in \Cref{sect:fuf}.
\end{rmk}

\subsection{Indecomposable elements in real quadratic fields}\label{sect:qi} 

The structure of the indecomposables in real quadratic fields is well understood. We give a brief overview of the theory here, for it is used in our proofs. All the results here can be found in \cite[Section~2.1]{BK2018}. Let $F = \Q(\sqrt{D})$ be a real quadratic field, where $D\ge 2$ is a squarefree integer. Define
\begin{equation}\label{eq:od} 
\omega_D := \begin{cases} \sqrt{D} & \text{ if } D\equiv 2,3\pmod{4},\\ \frac{1+\sqrt{D}}2 & \text{ if } D\equiv 1\pmod{4}.\end{cases}
\end{equation}
Then $\cb{1,\omega_D}$ forms a $\Z$-basis of $\mc O_F$. We denote by $\ol{\omega_D}$ the Galois conjugate of $\omega_D$. We know that $-\ol{\omega_D}$ has an eventually periodic continued fraction of the form
\begin{equation*}
-\ol{\omega_D} = \sqb{u_0;\ol{u_1,u_2,\ldots,u_{s-1},u_s}}.
\end{equation*}

For $i\in\N_0$, we define the $i$-th \emph{convergent} of $-\ol{\omega_D}$ to be
\ba
\frac{s_i}{t_i} = [u_0;u_1,\ldots,u_i].
\ea
As a convention, we shall also define $s_{-1} := 1$, $t_{-1} := 0$. By a \emph{semiconvergent} of $-\ol{\omega_D}$ we mean a fraction of the form 
\ba
\frac{s_{i,l}}{t_{i,l}} := \frac{s_i+ls_{i+1}}{t_i+lt_{i+1}},  
\ea
with $i\ge -1$ and $0\le l \le u_{i+2}$.

Next we define the elements $\alpha_i := s_i+t_i\omega_D$ and $\alpha_{i,l} := s_{i,l}+t_{i,l}\omega_D = \alpha_i + l\alpha_{i+1}$ in $\mc O_F$. By an abuse of notation, we shall also call $\alpha_i$ (resp. $\alpha_{i,l}$) the convergents (resp. semiconvergents) of $-\ol{\omega_D}$. By a classical theorem (see \cite[Theorem~2]{DS1982}), the set of indecomposable elements in $\mc O_F^+$ is then given precisely by $\alpha_{i,l}$ with $i\ge -1$ odd, and their Galois conjugates. The fundamental unit $\varepsilon$ of $\mc O_F$ is given by $\alpha_{s-1}$, which is totally positive if and only if $s$ is even. 

The set of indecomposable elements in $\mc O_F^+$ is closed under multiplication by totally positive units. A complete list of indecomposable elements of $\mc O_F^+$ up to multiplication by totally positive units is given by
\begin{equation}\label{eq:qil} 
\begin{aligned}
\cbm{\alpha_{i,l}}{-1 \le i < s-1 \text{ odd, } 0 \le l < u_{i+2}} \quad &\text{ if $s$ is even,}\\ \cbm{\alpha_{i,l}}{-1 \le i < 2s-1 \text{ odd, } 0 \le l < u_{i+2}} \quad &\text{ if $s$ is odd.}
\end{aligned}
\end{equation}

\section{Density results}\label{sect:den}

In this section, we prove \Cref{thm:main1,thm:main2,thm:fuf,thm:fi}. 

\subsection{Short lattice vectors and continued fractions} To state the proof of \Cref{thm:fuf}, we need the following theorem from \cite{RSDp2023}, which gives an upper bound to the number of vectors in a totally positive classical quadratic $\Z$-lattice with a given norm.

\begin{thm}[{\cite[Theorem 1.1]{RSDp2023}}]\label{thm:lpn} 
Let $(\Lambda,Q)$ be a positive definite classical quadratic $\Z$-lattice of rank $R$, and $m\in\N$. Let $N_{(\Lambda,Q)}(m)$ denote the number of vectors of norm $m$ in $\Lambda$, that is, the number of elements $v\in \Lambda$ such that $Q(v) = m$. Then we have
\begin{equation}\label{eq:lpnb} 
N_{(\Lambda,Q)}(m) \le 2 \binom{R+2m-2}{2m-1}.
\end{equation}
\end{thm}

Actually, the theory of root systems yields the following bound for $m=2$ \cite[Theorem 4.10.6, Proposition 4.10.7]{Martinet2003}, which is better than the bound \eqref{eq:lpnb} when $R\ge 11$:
\begin{equation}\label{eq:rsb} 
N_{(\Lambda,Q)}(2) \le \max\cb{480,2R(R-1)}.
\end{equation}
For convenience, we define for $R,m\in\N$
\begin{equation}\label{eq:crmd} 
C(R,m) := \begin{cases} \max\cb{480,2R(R-1)} & \text{ if } m=2,\\ 2 \binom{R+2m-2}{2m-1} & \text{ otherwise.}\end{cases}
\end{equation}

\begin{prp}\label{prp:ufcf} 
Let $n,m\in\N$, and $K = \Q(\sqrt{A_{2^0}}, \ldots, \sqrt{A_{2^{n-1}}})$ be a real multiquadratic field of degree $[K:\Q] = 2^n$, where $A_{2^0}, \ldots, A_{2^{n-1}}$ are squarefree positive integers. Let $\Q(\sqrt{A_j})$, $1\le j \le 2^n-1$, be the quadratic subfields of $K$ as constructed in \Cref{sect:mf}. Let $\omega_j := \omega_{A_j}$ be defined as in \eqref{eq:od}, and let
\ba
-\ol{\omega_j} = [u_{j,0}; \ol{u_{j,1},\ldots,u_{j,s_j}}]
\ea
be the periodic continued fraction of $-\ol{\omega_j}$, and $u := \max\cbm{u_{j,2i+1}}{1\le j \le 2^n-1,\; i\in\N_0}$. Suppose $(\Lambda,Q)$ is an $m\mc O_K$-universal classical quadratic $\mc O_K$-lattice of rank $R$. Then
\ba
u < \frac 12 C(2^nR, 2^{n-1}m),
\ea
where $C(R,m)$ is defined in \eqref{eq:crmd}.
\end{prp}
\begin{proof}
Suppose $u = u_{j,2i+1}$. As in \Cref{sect:qi}, we let $\frac{s_i}{t_i} = [u_{j,0}; u_{j,1},\ldots,u_{j,i}]$ be the $i$-th convergent of $-\ol{\omega_j}$, $\alpha_i = s_i+t_i\omega_j$, and consider the semiconvergents $\beta_l := \alpha_{2i-1}+l\alpha_{2i}$, $0\le l \le u_{2i+1} = u$. From \cite[Section 3]{KT2023}, we can find a totally positive element $\delta \in \mc O_{\Q(\sqrt{A_j})}^{\vee,+}$ in the codifferent such that $\Tr_{\Q(\sqrt{A_j})/\Q}(\delta \beta_l) =1$ for $0 \le l \le u$. It follows that $\Tr_{K/\Q}(\delta \beta_l) = 2^{n-1}$ for $0\le l \le u$. By fixing a $\Z$-basis of $\mc O_K$, we can identify $\Lambda$ with a classical $\Z$-lattice of rank $2^nR$ via an isomorphism $\varphi:\Lambda \xrightarrow{\sim} \Z^{2^nR}$, and equip it with a quadratic form $q(v):= \Tr_{K/\Q}(\delta Q(\varphi^{-1}(v)))$. Since $\delta \in \mc O_{\Q(\sqrt{A_j})}^{\vee,+}$ is totally positive, the quadratic form $q$ on $\Z^{2^nR}$ is positive definite.

From our assumptions, $Q$ represents all of $m\mc O_K^+$. In particular, we can find vectors $w_l \in \Lambda$ such that $Q(w_l) = m\beta_l$. For the corresponding vector $v_l := \varphi(w_l)$, we have $q(\pm v_r) = q(v_r) = \Tr_{K/\Q}(\delta Q(w_r)) = \Tr_{K/\Q}(m\delta \beta_l) = 2^{n-1}m$. Thus the number $N_{(\Z^{2^nR},q)}(2^{n-1}m)$ of vectors of norm $2^{n-1}m$ in the lattice $(\Z^{2^nR}, q)$ satisfies $N_{(\Z^{2^nR},q)}(2^{n-1}m) \ge 2(u+1) > 2u$. Finally, we use \eqref{eq:lpnb} and \eqref{eq:rsb} to conclude that $u < \frac 12 C(2^nR, 2^{n-1}m)$.
\end{proof}

We also make use of the following proposition from \cite{KYZp2023} about the behaviour of the continued fraction representation of $-\ol{\omega_D}$. 

\begin{prp}[{\cite[Corollary 2.13]{KYZp2023}}]\label{prp:scf} 
Let $B\ge 2$ be an integer, and $X\ge 2$ a parameter satisfying $X > B^4 (\log X)^4$. Then we have
\begin{multline*}
\#\cbm{1\le D \le X}{-\ol{\omega_D} \hspace{-0.2cm} \mod 1 = [0; u_1,u_2,\ldots], \; u_{2i-1} \le B \text{ for all } i\in\N}\\
< 50 B^{\frac 32} X^{\frac 78} (\log X)^{\frac 32} + 23 B^3 X^{\frac 34} (\log X)^2.
\end{multline*}
\end{prp}

\subsection{Proof of \texorpdfstring{\Cref{thm:fuf}}{Theorem 1.3}}\label{sect:fuf} 

For $X$ sufficiently large, we prove an upper bound to the number of multiquadratic fields $K$ with $\Delta_K \le X$ which admits an $m\mc O_K$-universal classical quadratic $\mc O_K$-lattice of rank $R$, using \Cref{prp:ufcf} and \Cref{prp:scf}. The case $n=1$ is essentially \cite[Theorem 1.2]{KYZp2023}, so we may assume $n\ge 2$. It is more convenient to count the four types of multiquadratic fields in \eqref{eq:mfa} separately, so the quantity $r$ defined in \eqref{eq:drd} remains stable in the following arguments; the arguments themselves do not depend on the type, however. So we fix a type, and bound the number of fields of that type satisfying the conditions above. We recall that for such multiquadratic fields $K$, the discriminant of $K$ is given by $\Delta_K = (2^r \rad(A_{2^0}\cdots A_{2^{n-1}}))^{2^{n-1}}$, where $r$ is given in \eqref{eq:drd}, depending on the type. 

For the counting, we need an alternative parametrisation of $K = \Q(\sqrt{A_{2^0}},\ldots,\sqrt{A_{2^{n-1}}})$. For $1\le j \le 2^{n-1}$, we write
\ba
j &= \sum_{k=0}^{n-1} 2^k \epsilon_{j,k}, &\quad  \epsilon_{j,k} &\in \cb{0,1}
\ea
for the binary representation of $j$. Let $p$ be a prime dividing $A_{2^0}\cdots A_{2^{n-1}}$, and let $\mc L(p)$ denote the subset of $\cb{0,1,\ldots,n-1}$ so that $p \mid A_{2^k}$ if and only if $k \in \mc L(p)$. Since $A_j$ is constructed as the squarefree part of the product $\prod_{\epsilon_{j,k}=1} A_{2^k}$, it follows that $p\mid A_j$ if and only if an odd number of the $A_{2^k}$'s in the product $\prod_{\epsilon_{j,k}=1} A_{2^k}$ are divisible by $p$. In other words, we have $p\mid A_j$ if and only if the binary representation of $j$ satisfies
\ba
\sum_{k\in\mc L(p)} \epsilon_{j,k} \equiv 1\pmod{2}. 
\ea
This gives a partition of the primes dividing $A_{2^0}\cdots A_{2^{n-1}}$ into $2^n-1$ classes, labelled by nonempty subsets of $\cb{0,1,\ldots,n-1}$, which correspond bijectively to the set of integers $\cb{1,\ldots,2^n-1}$ via the binary representation. So we may define pairwise coprime integers $\gamma_1,\ldots,\gamma_{2^n-1}\in\N$ by
\begin{equation*}\label{eq:gd} 
\gamma_i := \prod_{\substack{p \text{ prime}\\ \mc L(p) = \{k\;|\;\epsilon_{i,k} = 1\}}} p.
\end{equation*}
By construction, we have $\prod_{j=1}^{2^n-1} \gamma_j = \rad(A_{2^0}\cdots A_{2^{n-1}})$, and for $1\le i \le 2^n-1$, we have
\begin{equation}\label{eq:gc} 
A_i = \prod_{j\in\mc I(i)} \gamma_j, \quad \mc I(i) := \cbm{1\le j \le 2^n-1}{\sum_{\substack{0\le k \le n-1\\ \epsilon_{i,k} = 1}} \epsilon_{j,k} \equiv 1\pmod{2}}.
\end{equation}
In this parametrisation, the condition $\Delta_K \le X$ becomes 
\begin{equation}\label{eq:gp} 
\prod_{j=1}^{2^n-1} \gamma_j \le 2^{-r}X^{2^{1-n}} =: X_0.
\end{equation}

Let $\mc K'_X$ denote the set of multiquadratic fields $K$ of the given type, with discriminant $\Delta_K \le X$. Using \eqref{eq:gp}, the problem of finding $\#\mc K'_X$ reduces to that of counting tuples of pairwise coprime squarefree integers (ignoring order) with a bounded product and satisfy some congruence conditions. The number of tuples satisfying \eqref{eq:gp} can be estimated by the volume under the hyperplane $\prod_{j=1}^{2^n-1} \gamma_j = X_0$. Thus we obtain the asymptotic estimate
\begin{equation*}
\#\mc K'_X \asymp X_0 (\log X_0)^{2^n-2}.
\end{equation*} 
Note that this also gives an alternative proof of \Cref{thm:Wright}. 

Next we consider some dyadic blocks of multiquadratic fields. Our strategy is to cover most of $\mc K'_X$ with these dyadic blocks, and show that in each block there are few fields admitting $m\mc O_K$-universal quadratic lattices of rank $R$.

To define the dyadic blocks, let $Y\ll X_0$ be a sufficiently large parameter, and for $1\le j \le 2^n-2$, let $\Gamma_j \ge 1$ be parameters satisfying $\prod_{j=1}^{2^n-2} \Gamma_j\ll X_0$. Let $\mc K'(\Gamma_1,\ldots,\Gamma_{2^n-2};Y)$ denote the dyadic block of multiquadratic fields $K$ of the given type satisfying
\begin{equation}\label{eq:db} 
\frac{\Gamma_j}2 < \gamma_j \le \Gamma_j \quad (1\le j \le 2^n-2), \quad \text{ and } \quad \frac Y2 < \prod_{j=1}^{2^n-1} \gamma_j \le Y.
\end{equation}
The fields in the set $\mc K'(\Gamma_1,\ldots,\Gamma_{2^n-2};Y)$ are then represented by tuples $(\gamma_j)_{1\le j \le 2^n-1}$ satisfying \eqref{eq:db}. By counting the number of such integer tuples, we see that $\#\mc K'(\Gamma_1,\ldots,\Gamma_{2^n-2};Y) \asymp Y$. Writing $\Gamma_{2^n-1} := Y/\prod_{j=1}^{2^n-2}\Gamma_j$, the relations \eqref{eq:gc} say that for $K \in \mc K'(\Gamma_1,\ldots,\Gamma_{2^n-2};Y)$ we have
\begin{equation*}
A_i \asymp \prod_{j\in \mc I(i)} \Gamma_j.
\end{equation*}
Since $\gamma_j$ appears in exactly $2^{n-1}$ of the $A_i$'s, it follows that
\ba
\prod_{i=1}^{2^n-1} A_i \asymp \prod_{j=1}^{2^n-1} \Gamma_j^{2^{n-1}} = Y^{2^{n-1}}.
\ea
In particular, this implies that for any fixed set of parameters $\Gamma_1,\ldots,\Gamma_{2^n-2}$, we can find some $i$ such that we have
\begin{equation}\label{eq:ailb} 
\prod_{j\in \mc I(i)} \Gamma_j \gg Y^{\frac{2^{n-1}}{2^n-1}}.
\end{equation}
We fix this index $i$, and let $\mc A_i$ denote the set of possible values of $A_i$ for $K\in \mc K'(\Gamma_1,\ldots,\Gamma_{2^n-2};Y)$, namely
\ba
\mc A_i := \cbm{A_i = \prod_{j\in\mc I(i)} \gamma_j}{(\gamma_j)_{1\le j \le 2^n-1} \in \mc K'(\Gamma_1,\ldots,\Gamma_{2^n-2};Y)}.
\ea

Next we count, for fixed $A_i \in \mc A_i$, the number of tuples $(\gamma_j)_{1\le j \le 2^n-1}$ (and hence multiquadratic fields) in $\mc K'(\Gamma_1,\ldots,\Gamma_{2^n-2};Y)$ with $\prod_{j\in\mc I(i)} \gamma_j = A_i$. Since the divisor function $\tau(x)$ satisfies $\tau(x)\ll x^\varepsilon$ for every $\varepsilon > 0$ (see \cite[Theorem~317]{HW2008}), there are $\ll Y^\varepsilon$ tuples $(\gamma_j)_{j\in\mc I(i)}$ which satisfy $\prod_{j\in\mc I(i)} \gamma_j = A_i$. Meanwhile, the number of subtuples $(\gamma_j)_{j\not\in\mc I(i)}$ occuring in $\mc K'(\Gamma_1,\ldots,\Gamma_{2^n-2};Y)$ has size $\asymp \prod_{j\not\in\mc I(i)} \Gamma_j$. Therefore, for fixed $A_i\in\mc A_i$ we have
\begin{equation}\label{eq:aim} 
\#\cbm{(\gamma_j)_{1\le j \le 2^n-1} \in \mc K'(\Gamma_1,\ldots,\Gamma_{2^n-2};Y)}{\prod_{j\in\mc I(i)} \gamma_j = A_i} \ll Y^\varepsilon \prod_{j\not\in\mc I(i)} \Gamma_j.
\end{equation}

On the other hand, applying \Cref{prp:scf} (with $X = \prod_{j\in \mc I(i)} \Gamma_j$) yields for $B\ge 2$ and $\prod_{j\in \mc I(i)} \Gamma_j > B^4 (\sum_{j\in\mc I(i)} \log \Gamma_j)^4$ the bound
\begin{multline}\label{eq:aicf} 
\#\cbm{A_i\in \mc A_i}{-\ol{\omega_i} \hspace{-0.2cm} \mod 1 \equiv [0;u_1,u_2,\ldots], \; u_{2j-1}\le B \text{ for all } j\in\N}\\
\ll  B^{\frac 32} \prod_{j\in \mc I(i)} \Gamma_j^{\frac 78+\varepsilon} + B^3 \prod_{j\in \mc I(i)} \Gamma_j^{\frac 34+\varepsilon}.
\end{multline}
Combining \eqref{eq:ailb}, \eqref{eq:aim} and \eqref{eq:aicf}, we conclude for $Y\gg (B \log Y)^{\frac{2^n-1}{2^{n-3}}}$ that
\begin{align*}
&\#\cbm{K\in\mc K'(\Gamma_1,\ldots,\Gamma_{2^n-2};Y)}{-\ol{\omega_i} \hspace{-0.2cm} \mod 1 \equiv [0;u_1,u_2,\ldots], \; u_{2j-1}\le B \text{ for all } j\in\N}\\
&\hspace{2cm} \ll \rb{Y^\varepsilon \prod_{j\not\in\mc I(i)} \Gamma_j} \rb{B^{\frac 32} \prod_{j\in \mc I(i)} \Gamma_j^{\frac 78+\varepsilon} + B^3 \prod_{j\in \mc I(i)} \Gamma_j^{\frac 34+\varepsilon}}\\
&\hspace{2cm} \ll B^{\frac 32} Y^{1+\varepsilon} \prod_{j\in\mc I(i)} \Gamma_j^{-\frac 18+\varepsilon} + B^3 Y^{1+\varepsilon} \prod_{j\in\mc I(i)} \Gamma_j^{-\frac 14+\varepsilon}\\
&\hspace{2cm} \ll B^{\frac 32} Y^{1-\frac{2^{n-4}}{2^n-1}+\varepsilon} + B^3 Y^{1-\frac{2^{n-3}}{2^n-1}+\varepsilon}.
\end{align*}

Finally, we observe that the fields in $\mc K'_X$ with discriminant $X^{1-\frac{2^{n-4}}{2^n-1}} < \Delta_K \le X$ can be covered by $\asymp \log(X_0)^{2^n-2}$ such dyadic blocks by varying the parameters $\Gamma_j$ and $Y$; there are $\ll X_0^{1-\frac{2^{n-4}}{2^n-1}} \log(X_0)^{2^n-2}$ fields remaining with discriminant $\Delta_K \le X^{1-\frac{2^{n-4}}{2^n-1}}$. It follows that for $X_0 \gg (B \log X_0)^{\frac{(2^n-1)^2}{2^{n-3}(2^n-2^{n-4}-1)}}$ we have
\begin{multline*}
\#\cbm{K\in\mc K'_X}{-\ol{\omega_i} \hspace{-0.2cm} \mod 1 \equiv [0;u_1,u_2,\ldots], \; u_{2j-1}\le B \fa j\in\N, 1\le i \le 2^n-1}\\
\ll B^{\frac 32} X_0^{1-\frac{2^{n-4}}{2^n-1}+\varepsilon} + B^3 X_0^{1-\frac{2^{n-3}}{2^n-1}+\varepsilon}.
\end{multline*}
\Cref{thm:fuf} then follows once we put $B = \frac 12 C(2^nR, 2^{n-1}m)$ and apply \Cref{prp:ufcf}.

\subsection{Number of indecomposable elements}

For a totally real number field $F$, the minimal rank of a universal $\mc O_F$-lattice gives a convenient lower bound to the number of indecomposable elements $\iota(F)$. Using \Cref{thm:fuf}, this gives a proof that there are few multiquadratic fields with few indecomposable elements.

To state the proof of \Cref{thm:fi}, we recall that the \emph{Pythagoras number} of a ring $R$ is defined as the smallest integer $s(R)$ such that every sum of squares of elements of $R$ can be expressed as the sum of $s(R)$ squares; if such an integer does not exist, we set $s(R) := \infty$. For a totally real number field $F$, it is well-known that the Pythagoras number $s(\mc O_F)$ is finite, and that $s(\mc O_F)\le f(d)$ where $f(d)$ is  a function that depends only on the degree $d = [F:\Q]$ of $F$ \cite[Corollary 3.3]{KY2021}. The following proposition gives a construction of a classical universal $\mc O_F$-lattice, using indecomposable elements of $\mc O_F$.

\begin{prp}[{\cite[Proposition 7.1]{KT2023}}]\label{prp:ufi} 
Let $F$ be a totally real number field, and $s = s(\mc O_F)$ the Pythagoras number of $\mc O_F$. Let $\mc S = \mc S(\mc O_F)$ denote a set of indecomposable elements in $\mc O_F^+$ up to multiplication by squares of units $(\mc O_F^\times)^2$. Then the diagonal quadratic form
\ba
\sum_{\sigma \in \mc S} \sigma \rb{x_{1,\sigma}^2+x_{2,\sigma}^2+\ldots+x_{s,\sigma}^2}
\ea
defines a universal $\mc O_F$-lattice of rank $s \cdot \#\mc S$.
\end{prp}

\begin{proof}[Proof of \Cref{thm:fi}]
We apply \Cref{prp:ufi} to our multiquadratic field $K$. Consider the inclusion $(\mc O_K^\times)^2 \sbe \mc O_K^{\times,+} \sbe \mc O_K^\times$. Since we have $[\mc O_K^\times : (\mc O_K^\times)^2] = 2^{[K:\Q]} = 2^{2^n}$ and $[\mc O_K^\times:\mc O_K^{\times,+}] \ge 2$ (because of the element $-1$), it follows that the set $\mc S(\mc O_K)$ has size $\#\mc S(\mc O_K) \le 2^{2^n-1} \iota(K)$. On the other hand, we have seen that the Pythagoras number $s(\mc O_K)$ is bounded by a function depending only on $n$. From \Cref{prp:ufi}, it follows that
\ba
R_{\operatorname{cls}}(K) \le s(\mc O_K) \cdot \#\mc S(\mc O_K) \ll_n \iota(K).
\ea
The theorem then follows immediately from \Cref{thm:fuf}.
\end{proof}

\subsection{Proof of \texorpdfstring{\Cref{thm:main1,thm:main2}}{Theorems 1.1 and 1.2}} Having established \Cref{thm:fuf,thm:fi}, we are ready to prove \Cref{thm:main1,thm:main2}.
\begin{proof}[Proof of \Cref{thm:main1}]
Let $X$ be a sufficiently large parameter. To obtain density results for $R_{\operatorname{cls}}(K)$, we choose the parameter $R$ such that \Cref{thm:fuf} gives
\ba
\scalebox{0.97}{$\displaystyle \mc K_{\operatorname{univ}}(n,X,R,1) \ll_{n,\varepsilon} C(2^nR,2^{n-1})^{\frac 32} X^{2^{1-n}\rb{1-\frac{2^{n-4}}{2^n-1}}+\varepsilon} + C(2^nR,2^{n-1})^3 X^{2^{1-n}\rb{1-\frac{2^{n-3}}{2^n-1}}+\varepsilon} \ll X^{2^{1-n}-\varepsilon}.$}
\ea
From \eqref{eq:crmd}, we see that for fixed $m\in\N$, we have as $R\to\infty$
\begin{equation}\label{eq:crma} 
C(R,m) \ll \begin{cases} R^2 & \text{ if } m=2,\\ R^{2m-1} & \text{ otherwise.} \end{cases}
\end{equation}
It follows that we may take $R = X^{\frac 1{72}-\varepsilon}$ when $n=2$, and $R = X^{\frac 1{12(2^n-1)^2}-\varepsilon}$ otherwise. Note that for these choices, the condition $X \gg (C(2^nR, 2^{n-1}) \log X)^{\frac{4(2^n-1)^2}{2^n-2^{n-4}-1}}$ in \Cref{thm:fuf} is always satisfied. This proves the density results for $R_{\operatorname{cls}}(K)$.

Meanwhile, to obtain density results for $R(K)$, we note that if $(\Lambda,Q)$ is a (possibly) non-classical universal lattice, then $(\Lambda,2Q)$ is a classical $2\mc O_K$-universal lattice of the same rank. Hence, in this case we choose the parameter $R$ such that \Cref{thm:fuf} gives
\ba
\mc K_{\operatorname{univ}}(n,X,R,2) \ll_{n,\varepsilon} C(2^nR,2^n)^{\frac 32} X^{2^{1-n}\rb{1-\frac{2^{n-4}}{2^n-1}}+\varepsilon} + C(2^nR,2^n)^3 X^{2^{1-n}\rb{1-\frac{2^{n-3}}{2^n-1}}+\varepsilon} \ll X^{2^{1-n}-\varepsilon}.
\ea
By \eqref{eq:crma}, we may take $R = X^{\frac 1{24}-\varepsilon}$ when $n=1$, and $R = X^{\frac 1{12(2^n-1)(2^{n+1}-1)}-\varepsilon}$ otherwise. Again, for these choices the condition $X \gg (C(2^nR, 2^n) \log X)^{\frac{4(2^n-1)^2}{2^n-2^{n-4}-1}}$ in \Cref{thm:fuf} is always satisfied. This proves the density results for $R(K)$.
\end{proof}

\begin{proof}[Proof of \Cref{thm:main2}]
The proof is completely analogous to that of \Cref{thm:main1}, the only difference being that we use \Cref{thm:fi} instead of \Cref{thm:fuf}.
\end{proof}

\section{Preservation of indecomposability}\label{sect:pi} 

In \Cref{sect:pi,sect:if}, we focus on biquadratic fields. First we establish the notations used in these sections.

\subsection{The setup}\label{sect:bf} 
Let $p$ and $q$ to be distinct squarefree positive integers, and $K := \Q(\sqrt{p},\sqrt{q})$ a biquadratic field. Given $p$ and $q$, we fix $r = \frac{pq}{\gcd(p,q)^2}$. The biquadratic field $K$ has three quadratic subfields, namely $K_p := \Q(\sqrt{p})$, $K_q := \Q(\sqrt{q})$, and $K_r := \Q(\sqrt{r})$. There are four embeddings of $K$ into $\R$. Writing $\alpha = x+y\sqrt{p}+z\sqrt{q}+w\sqrt{r} \in K$, with $x,y,z,w\in\Q$, the embeddings are given as follows:
\begin{align*}
\sigma_1(\alpha) &= x+y\sqrt{p}+z\sqrt{q}+w\sqrt{r}, & \sigma_2(\alpha) &= x-y\sqrt{p}+z\sqrt{q}-w\sqrt{r},\\
\sigma_3(\alpha) &= x+y\sqrt{p}-z\sqrt{q}-w\sqrt{r}, & \sigma_4(\alpha) &= x-y\sqrt{p}-z\sqrt{q}+w\sqrt{r}.
\end{align*}
Swapping $p,q,r$ if necessary, every biquadratic field belongs to exactly one of the following types: 
\ben
\item $p\equiv 2 \pmod{4}$, $q\equiv 3\pmod{4}$,
\item $p\equiv 2,3 \pmod{4}$, $q\equiv 1\pmod{4}$,
\item $p\equiv q\equiv 1\pmod{4}$, $\gcd(p,q) \equiv 1\pmod{4}$,
\item $p\equiv q\equiv 1\pmod{4}$, $\gcd(p,q) \equiv 3\pmod{4}$.
\ee
Note that in all cases, we have $p\equiv r\pmod{4}$, so $p$ and $r$ are interchangeable, and we may always assume $p<r$. Moreover, for biquadratic fields of types (3) and (4), $p,q,r$ are all interchangeable, and we may further assume $p<q<r$. For each of the types above, an integral basis of $\mc O_K$ is given by (also see \cite[Theorem 2]{Williams1970})
\ben
\item $\cb{1,\sqrt{p},\sqrt{q}, \frac{\sqrt{p}+\sqrt{r}}{2}}$,
\item $\cb{1,\sqrt{p},\frac{1+\sqrt{q}}{2}, \frac{\sqrt{p}+\sqrt{r}}{2}}$,
\item $\cb{1,\frac{1+\sqrt{p}}{2},\frac{1+\sqrt{q}}{2}, \frac{1+\sqrt{p}+\sqrt{q}+\sqrt{r}}{4}}$,
\item $\cb{1,\frac{1+\sqrt{p}}{2},\frac{1+\sqrt{q}}{2}, \frac{1-\sqrt{p}+\sqrt{q}+\sqrt{r}}{4}}$.
\ee

\subsection{Diophantine approximations}
For the proof of \Cref{thm:pii}, we also need some results on Diophantine approximations. 

Let $\alpha \in \R$, and $\frac st \in \Q$, with $s\in\Z$, $t\in\N$. We say $\frac st$ is a \emph{best Diophantine approximation of the second kind} of $\alpha$ if we have
\ba
\vb{s-t\alpha} < \vb{s'-t'\alpha}
\ea
for all $(s',t') \ne (s,t)$, with $s'\in\Z$, $t'\in\N$, $t'\le t$. In a similar fashion, we consider one-sided Diophantine approximations. We say $\frac st$ is a \emph{best upper bound of the second kind} of $\alpha$ if we have
\ba
s'-t'\alpha > 0 \implies 0 < s-t\alpha < s'-t'\alpha
\ea
for all $(s',t') \ne (s,t)$, with $s'\in\Z$, $t'\in\N$, $t'\le t$. Analogously, we say $\frac st$ is a \emph{best lower bound of the second kind} of $\alpha$ if we have
\ba
s'-t'\alpha < 0 \implies s'-t'\alpha < s-t\alpha < 0
\ea
for all $(s',t') \ne (s,t)$, with $s'\in\Z$, $t'\in\N$, $t'\le t$. 
\begin{rmk}
In the following arguments, we often say that a fraction $\frac{s'}{t'}$ is a ``better Diophantine approximation of the second kind'' of $\alpha$ than another fraction $\frac st$; by this we mean $\vb{s'-t'\alpha} < \vb{s-t\alpha}$. In particular, when $\frac st$ is a best Diophantine approximation of the second kind of $\alpha$, then this implies $t'>t$. The phrases ``better upper (resp. lower) bound of the second kind'' are interpreted analogously.
\end{rmk}

These Diophantine approximations are well understood in terms of continued fractions. A survey in these topics can be found in \cite[Chapter~II]{Perron1977} and \cite[Section~4]{HT2019}. Let
\ba
\alpha = [u_0; u_1,u_2,\ldots]
\ea
be the continued fraction representation of $\alpha$. For $i\in\N_0$, the $i$-th \emph{convergent} of $\alpha$ is given by
\ba
\frac{s_i}{t_i} := [u_0;u_1,\ldots,u_i].
\ea
As a convention, we shall also define $s_{-1} := 1$, $t_{-1} := 0$. Note that for $i\in\N_0$ we have $\frac{s_i}{t_i} \ge \alpha$ when $i$ is odd, and $\frac{s_i}{t_i} \le \alpha$ when $i$ is even. Accordingly, we call these the \emph{upper} (resp. \emph{lower}) \emph{convergents} of $\alpha$ respectively. The numbers $s_i,t_i$ satisfy the relation
\begin{equation}\label{eq:cm1} 
s_{i-1}t_i - s_it_{i-1} = (-1)^i, \quad i\ge 0.
\end{equation}

The \emph{semiconvergents} of $\alpha$ are given by
\ba
\frac{s_{i,l}}{t_{i,l}} := \frac{s_i+ls_{i+1}}{t_i+lt_{i+1}}, \qquad (0\le l \le u_{i+2}-1).
\ea 
Since the semiconvergents $\frac{s_{i,l}}{t_{i,l}}$ lie between $\frac{s_i}{t_i}$ and $\frac{s_{i+2}}{t_{i+2}}$, we have $\frac{s_{i,l}}{t_{i,l}} \ge \alpha$ when $i$ is odd, and $\frac{s_{i,l}}{t_{i,l}} \le \alpha$ when $i$ is even. Accordingly, we call these the \emph{upper} (resp. \emph{lower}) \emph{semiconvergents} of $\alpha$ respectively. Note that we always have $\gcd(s_{i,l},t_{i,l}) = 1$. We have the following well-known theorems (see \cite[Theorem~4.5]{HT2019}).
\begin{thm}\label{thm:bb} 
\ben
\item The set of best Diophantine approximation of the second kind of $\alpha$ is given precisely by the set of convergents $\frac{s_i}{t_i}$ for $i\ge 0$.
\item The set of best upper bounds of the second kind of $\alpha$ is given precisely by the set of upper semiconvergents $\frac{s_{i,l}}{t_{i,l}}$ for $i\ge -1$ odd, except for the pair $(i,l) = (-1,0)$. 
\item The set of best lower bounds of the second kind of $\alpha$ is given precisely by the set of lower semiconvergents $\frac{s_{i,l}}{t_{i,l}}$ for $i\ge 0$ even.
\ee
\end{thm}
\begin{cor}\label{cor:bbc} 
Let $\frac{s_{i,l}}{t_{i,l}}$ be a semiconvergent of $\alpha$ that is not a convergent. If $s'\in\Z$, $t'\in\N$, $t' \le t_{i,l}$ satisfy
\begin{equation*}
\vb{s'-t'\alpha} < \vb{s_{i,l}-t_{i,l}\alpha},
\end{equation*}
then $\frac{s'}{t'} = \frac{s_{i+1}}{t_{i+1}}$. 
\end{cor}
\begin{proof}
Suppose $\frac{s_{i,l}}{t_{i,l}}$ is an upper semiconvergent. From \Cref{thm:bb}, $\frac{s_{i,l}}{t_{i,l}}$ is a best upper bound of the second kind. This implies $\frac{s'}{t'}\le \alpha$ is a lower bound. Since $\frac{s'}{t'}$ is a better Diophantine approximation of the second kind of $\alpha$ than the $i$-th convergent $\frac{s_i}{t_i}$, we have $t' \ge t_i$. But there is a unique best lower bound of the second kind of $\alpha$ with denominator between $t_i$ and $t_{i,l} = t_i + lt_{i+1}$, namely the $(i+1)$-th convergent $\frac{s_{i+1}}{t_{i+1}}$. The proof for the case of a lower semiconvergent is completely analogous.
\end{proof}

Finally, we note that the notions of convergents and semiconvergents in \Cref{sect:qi} match the constructions of convergents and semiconvergents here. This allows us to apply techniques from Diophantine approximations to study indecomposable elements and prove \Cref{thm:pii}. 

\subsection{Proof of \texorpdfstring{\Cref{thm:pii}}{Theorem 1.5}}
Here we prove \Cref{thm:pii}, in the form of the following theorem.

\begin{thm}\label{thm:pid} 
Let $p,q$ be squarefree positive integers satisfying the congruences given in \Cref{sect:bf}. Let $K = \Q(\sqrt{p},\sqrt{q})$ be a real biquadratic field, and $r = \frac{pq}{\gcd(p,q)^2}$. 
\ben
\item Suppose $K$ is a biquadratic field of type (1) or (2), and $p<r$. Then the indecomposable elements in $\mc O_{K_p}^+$ and $\mc O_{K_q}^+$ remain indecomposable in $\mc O_K^+$.
\item Suppose $K$ is a biquadratic field of type (3) or (4), and $p<q<r$. Then the indecomposable elements in $\mc O_{K_p}^+$ and $\mc O_{K_q}^+$ remain indecomposable in $\mc O_K^+$.
\ee
\end{thm}

We shall prove \Cref{thm:pid} using the \Cref{prp:pi12,prp:pi34,prp:qi}, which deal with the cases separately.
\begin{prp}\label{prp:pi12} 
Let $K$ be a biquadratic field of type (1) or (2), and $p<r$. Then the indecomposable elements in $\mc O_{K_p}^+$ remain indecomposable in $\mc O_K^+$.  
\end{prp}
\begin{proof}
For a biquadratic field $K$ of these types, we have $\omega_p = \sqrt{p}$, so $\mc O_{K_p} = \Z[\sqrt{p}]$. Let $\beta = x+y\sqrt{p} \in \mc O_{K_p}^+$ be an indecomposable element in $\mc O_{K_p}^+$. We shall exclude the trivial case $\beta = 1$, which is always indecomposable. By taking Galois conjugates, we may assume $y>0$. Note that this implies $\ol{\beta} = x-y\sqrt{p} \le 1$. We also note that $\frac xy$ is a best upper bound of the second kind of $\sqrt{p}$; this follows from the characterisation of indecomposable elements (see \Cref{sect:qi}) and \Cref{thm:bb}. Suppose $\beta$ admits in $\mc O_K^+$ a decomposition
\begin{equation*}
\beta = \beta_1+\beta_2, \qquad  \beta_i = x_i + y_i\sqrt{p} + z_i\sqrt{q} + w_i\sqrt{r} \in \mc O_K^+.
\end{equation*}
Clearly we have $z_2 = -z_1$, and $w_2 = -w_1$. From the shape of the integral basis of $\mc O_K$ given in \Cref{sect:bf}, we see that $x_i,y_i,z_i,w_i\in\frac 12\Z$. We consider the equation
\begin{equation*}
2\ol{\beta} = 2x-2y\sqrt{p} = (2x_1-2y_1\sqrt{p}) + (2x_2-2y_2\sqrt{p}) = \sigma_2(\beta_1+\beta_2)+\sigma_4(\beta_1+\beta_2).
\end{equation*}
Since $\beta_1,\beta_2$ are totally positive, it follows that both $E_1:= 2x_1-2y_1\sqrt{p}$ and $E_2 := 2x_2-2y_2\sqrt{p}$ are positive. Without loss of generality, we may assume $2x_1-2y_1\sqrt{p} \le x-y\sqrt{p}$. Since $\frac xy$ is a best upper bound of the second kind of $\sqrt{p}$, this implies $2y_1 \ge y$, and $2x_1\ge x$. 

Now we claim that $2x_1+2y_1\sqrt{p}$ is indecomposable in $\mc O_{K_p}^+$. Suppose to the contrary that there is a decomposition
\ba
2x_1+2y_1\sqrt{p} = (x_{1a}+y_{1a}\sqrt{p}) + (x_{1b}+y_{1b}\sqrt{p}).
\ea
Since $\frac xy$ is a best upper bound of the second kind of $\sqrt{p}$, and
\ba
x_{1a}-y_{1a}\sqrt{p}, \; x_{1b}-y_{1b}\sqrt{p} < 2x_1-2y_1\sqrt{p} \le x-y\sqrt{p}, 
\ea
we deduce that $y_{1a},y_{1b}\ge y$, and hence $2y_1 \ge 2y$, and $2x_1\ge 2x$. But this implies $2x_2\le 0$, which contradicts to the assumption that $\beta_2$ is totally positive. So the claim is established. From the characterisation of the indecomposable elements in $\mc O_{K_p}^+$, this says $2x_1+2y_1\sqrt{p} = \alpha_{2i-1,l}$ is an upper semiconvergent of $\sqrt{p}$.

Now we find the possible decompositions of $\beta = x+y\sqrt{p}$ in $\mc O_K^+$, that is, find $z_1$ and $w_1$. From the shape of the integral basis of $\mc O_K$ given in \Cref{sect:bf}, we see that
\ber
\item $x_1+z_1, y_1+w_1\in\Z$.
\ee
Meanwhile, $\beta_1$ being totally positive implies
\ber\setcounter{enumi}{1}
\item $\vb{z_1\sqrt{q}} < x_1$, $\vb{w_1\sqrt{r}} < x_1$, and
\item $\vb{2z_1\sqrt{q}-2w_1\sqrt{r}} < 2x_1-2y_1\sqrt{p} = E_1$.
\ee
Using condition (iii), we see that $z_1$ and $w_1$ have the same sign, and we may assume both of them are positive. Rewriting condition (iii) yields the condition
\begin{align*}
\vb{2z_1\sqrt{q}-2w_1\sqrt{r}} < E_1 &\iff \vb{2z_1-2w_1\frac{\sqrt{r}}{\sqrt{q}}} = \vb{2z_1-2w_1\frac{\sqrt{p}}{\gcd(p,q)}} < \frac{E_1}{\sqrt{q}}\\
&\iff \vb{2\gcd(p,q)z_1-2w_1\sqrt{p}} < \frac{\gcd(p,q)E_1}{\sqrt{q}}.
\end{align*}
The assumption $p<r$ implies $\gcd(p,q) < \sqrt{q}$, so we have $|2\gcd(p,q)z_1-2w_1\sqrt{p}|< E_1$. On the other hand, condition (ii) says $2w_1\sqrt{r} < 2x_1 \le 2y_1\sqrt{p}+1$. Since $p<r$, and $y_1,w_1\in\frac 12\Z$, it follows that we have $w_1 \le y_1$. This says $\frac{2\gcd(p,q)z_1}{2w_1}$ is a better Diophantine approximation of the second kind of $\sqrt{p}$ than $\frac{2x_1}{2y_1}$, with $2w_1 \le 2y_1$. If $2x_1+2y_1\sqrt{p} = \alpha_{2i-1}$ is a convergent, then such $z_1, w_1$ cannot exist, and the indecomposability of $\beta$ in $\mc O_K^+$ is established. If $2x_1+2y_1\sqrt{p} = \alpha_{2i-1,l}$ is not a convergent, we use \Cref{cor:bbc} to deduce that $2\gcd(p,q)z_1+2w_1\sqrt{p} = k\alpha_{2i}$ is a multiple of the convergent $\alpha_{2i}$.

Writing $\alpha_j = s_j+t_j\sqrt{p}$ for the convergents of $\sqrt{p}$, we have
\begin{equation}\label{eq:pi1z}
2\gcd(p,q)z_1+2w_1\sqrt{p} = k\alpha_{2i} = ks_{2i}+kt_{2i}\sqrt{p},
\end{equation}
and
\ba
2x_1+2y_1\sqrt{p} = \alpha_{2i-1,l} = (s_{2i-1}+ls_{2i}) + (t_{2i-1}+lt_{2i})\sqrt{p}.
\ea

Suppose the biquadratic field $K$ is of type (1). From the shape of the integral basis of $\mc O_K$ given in \Cref{sect:bf}, we see that $x_1,z_1\in\Z$. Since we have $\gcd(2x_1,2y_1) = 1$, it follows that $y_1\in\frac 12+\Z$, and thus $w_1\in\frac 12+\Z$ by condition (i). It follows that $2w_1 = kt_{2i}$ is odd, and $2\gcd(p,q)z_1 = ks_{2i}$ is even. This implies $s_{2i}$ is even. Using \eqref{eq:cm1}, we deduce that $s_{2i-1}$ is odd. But then $2x_1 = s_{2i-1}+ls_{2i}$ is odd, a contradiction. So $\beta$ is indecomposable in $\mc O_K^+$ in this case.

Now suppose the biquadratic field $K$ is of type (2). Since $\gcd(2x_1,2y_1)=1$, it follows that $x_1,y_1$ cannot both be integers. From condition (i) we see that there are three possibilities.
\ber
\item Suppose $x_1,z_1\in\Z$, and $y_1,w_1\in\frac 12+\Z$. Then the argument above also applies, so $\beta$ is indecomposable in $\mc O_K^+$.
\item Suppose $x_1,z_1\in\frac 12+\Z$, and $y_1,w_1\in\Z$. Since $\gcd(p,q)$ is odd, it follows from \eqref{eq:pi1z} that $s_{2i}$ is odd, and $t_{2i}$ is even. Using \eqref{eq:cm1} we deduce that $t_{2i-1}$ is odd. This implies $2y_1 = t_{2i-1}+lt_{2i}$ is odd, a contradiction. So $\beta$ is indecomposable in $\mc O_K^+$.
\item Suppose $x_1,y_1,z_1,w_1\in\frac 12+\Z$. Since $\gcd(p,q)$ is odd, it follows from \eqref{eq:pi1z} that both $s_{2i}, t_{2i}$ are odd. Using \eqref{eq:cm1}, we see that exactly one of $s_{2i-1}, t_{2i-1}$ is even. But then $2x_1 = s_{2i-1}+ls_{2i}$ and $2y_1 = t_{2i-1}+lt_{2i}$ cannot both be odd, a contradiction. So $\beta$ is indecomposable in $\mc O_K^+$.
\ee
This finishes the proof for biquadratic fields of type (2).
\end{proof}

\begin{prp}\label{prp:pi34} 
Let $K$ be a biquadratic field of type (3) or (4), and $p<r$. Then the indecomposable elements in $\mc O_{K_p}^+$ remain indecomposable in $\mc O_K^+$.  
\end{prp}
\begin{proof}
For a biquadratic field $K$ of these types, we have $\omega_p = \frac{1+\sqrt{p}}2$, so $\mc O_{K_p} = \Z[\frac{1+\sqrt{p}}2]$. Let $\beta = x+y\frac{1+\sqrt{p}}2$ be an indecomposable element in $\mc O_{K_p}^+$. Again we shall exclude the trivial case $\beta = 1$. By taking Galois conjugates, we may assume $y>0$.  We also note that $\frac xy$ is a best upper bound of the second kind of $-\ol\omega_p = \frac{-1+\sqrt{p}}2$; this follows from the characterisation of indecomposable elements (see \Cref{sect:qi}) and \Cref{thm:bb}. Note that this implies $\ol\beta = x+y\frac{1-\sqrt{p}}2 \le 1$. Suppose $\beta$ admits in $\mc O_K^+$ a decomposition
\begin{equation*}
\beta = \beta_1+\beta_2, \qquad  \beta_i = x_i + y_i\frac{1+\sqrt{p}}2 + z_i\sqrt{q} + w_i\sqrt{r} \in \mc O_K^+.
\end{equation*}
Again we have $z_2 = -z_1$, and $w_2 = -w_1$. From the shape of the integral basis of $\mc O_K$ given in \Cref{sect:bf}, we see that $x_i, y_i \in \frac 12\Z$, and $z_i,w_i\in\frac 14\Z$. We consider the equation
\ba
2\ol\beta &= 2x+2y\frac{1-\sqrt{p}}2 = \rb{2x_1+2y_1\frac{1-\sqrt{p}}2} + \rb{2x_2+2y_2\frac{1-\sqrt{p}}2} = \sigma_2(\beta_1+\beta_2) + \sigma_4(\beta_1+\beta_2).
\ea
Since $\beta_1,\beta_2$ are totally positive, it follows that both $E_1 := 2x_1+2y_1\frac{1-\sqrt{p}}2$ and $E_2 := 2x_2+2y_2\frac{1-\sqrt{p}}2$ are positive. Without loss of generality, we may assume $2x_1+2y_1\frac{1-\sqrt{p}}2 \le x+y\frac{1-\sqrt{p}}2$. Since $\frac xy$ is a best upper bound of the second kind of $\frac{-1+\sqrt{p}}2$, this implies $2y_1\ge y$, and $2x_1\ge x$. Using the same argument as above for $K$ of types (1) and (2), we conclude that $2x_1+2y_1\frac{1+\sqrt{p}}2$ is indecomposable in $\mc O_{K_p}^+$. This says $2x_1+2y_1\frac{1+\sqrt{p}}2 = \alpha_{2i-1,l}$ is an upper semiconvergent of $\frac{-1+\sqrt{p}}2$. 

Now we find the possible decompositions of $\beta = x+y\frac{1+\sqrt{p}}2$ in $\mc O_K^+$, that is, find $z_1$ and $w_1$. Total positivity of $\beta_1$ implies
\ber
\item $\vb{z_1\sqrt{q}} < x_1+\frac{y_1}2$, $\vb{w_1\sqrt{r}} < x_1+\frac{y_1}2$, and
\item $\vb{2z_1\sqrt{q}-2w_1\sqrt{r}} < 2x_1+2y_1\frac{1-\sqrt{p}}2 = E_1$.
\ee
Using condition (ii), we see that $z_1$ and $w_1$ have the same sign, and we may assume both of them are positive. Rewriting condition (ii) yields the condition
\begin{multline*}
\vb{2z_1\sqrt{q}-2w_1\sqrt{r}} < E_1 \iff \vb{2z_1-2w_1\frac{\sqrt{r}}{\sqrt{q}}} = \vb{2z_1-2w_1\frac{\sqrt{p}}{\gcd(p,q)}} < \frac{E_1}{\sqrt{q}}\\
\iff \vb{2\gcd(p,q)z_1-2w_1\sqrt{p}} = \vb{(2\gcd(p,q)z_1-2w_1)+4w_1\frac{1-\sqrt{p}}2} < \frac{\gcd(p,q)E_1}{\sqrt{q}}.
\end{multline*}
The assumption $p<r$ implies $\gcd(p,q)<\sqrt{q}$, so we have $|(2\gcd(p,q)z_1-2w_1)+4w_1\frac{1-\sqrt{p}}2| < E_1$. On the other hand, using condition (i), we have
\ba
2w_1\sqrt{r} < 2x_1+y_1 = \rb{2x_1+2y_1\frac{1-\sqrt{p}}2} + y_1\sqrt{p} \le y_1\sqrt{p}+1.
\ea
Since $p<r$, and $y_1,2w_1\in\frac 12\Z$, it follows that we have $2w_1\le y_1$. Noting that $2\gcd(p,q)z_1-2w_1$ is always an integer, we see that $\frac{2\gcd(p,q)z_1-2w_1}{4w_1}$ is a better Diophantine approximation of the second kind of $\frac{-1+\sqrt{p}}2$ than $\frac{2x_1}{2y_1}$, with $2w_1\le 2y_1$. If $2x_1+2y_1\frac{1+\sqrt{p}}2 = \alpha_{2i-1}$ is a convergent, then such $z_1,w_1$ cannot exist, and the indecomposability of $\beta$ in $\mc O_K^+$ is established. If $2x_1+2y_1\frac{1+\sqrt{p}}2 = \alpha_{2i-1,l}$ is not a convergent, we use \Cref{cor:bbc} to deduce that $(2\gcd(p,q)z_1-2w_1)+4w_1\frac{1+\sqrt{p}}2 = k\alpha_{2i}$ is a multiple of the convergent $\alpha_{2i}$.

Writing $\alpha_j = s_j+t_j\frac{1+\sqrt{p}}2$ for the convergents of $\frac{-1+\sqrt{p}}2$, we have
\begin{equation}\label{eq:pi3z} 
(2\gcd(p,q)z_1-2w_1)+4w_1\frac{1+\sqrt{p}}2 = k\alpha_{2i} = ks_{2i}+kt_{2i}\frac{1+\sqrt{p}}2,
\end{equation}
and
\ba
2x_1+2y_1\frac{1+\sqrt{p}}2 = \alpha_{2i-1,l} = (s_{2i-1}+ls_{2i}) + (t_{2i-1}+lt_{2i})\frac{1+\sqrt{p}}2.
\ea

Suppose the biquadratic field $K$ is of type (3). Then we have $\gcd(p,q) \equiv 1\pmod{4}$. Since $\gcd(2x_1,2y_1)=1$, it follows that $x_1,y_1$ cannot both be integers, so there are three possibilities.
\ber
\item Suppose $x_1\in\Z$, $y_1\in\frac 12+\Z$. In this case, $\beta_1\in\mc O_K$ implies $z_1,w_1\in\frac 14+\frac 12\Z$, and $z_1-w_1\in\Z$. Using \eqref{eq:pi3z} and that $\gcd(p,q) \equiv 1\pmod{4}$, we deduce that $2\gcd(p,q)z_1-2w_1 = ks_{2i}$ is even, and $4w_1 = kt_{2i}$ is odd. This implies $t_{2i}$ is odd, and $s_{2i}$ is even. Using \eqref{eq:cm1}, we deduce that $s_{2i-1}$ is odd. This implies $2x_1=s_{2i-1}+ls_{2i}$ is odd, a contradiction. So $\beta$ is indecomposable in $\mc O_K^+$.
\item Suppose $x_1\in\frac 12\Z$, $y_1\in\Z$. In this case, $\beta_1\in\mc O_K$ implies $z_1,w_1\in\frac 12\Z$, and $z_1-w_1\in\frac 12+\Z$. Using \eqref{eq:pi3z}, we deduce that $2\gcd(p,q)z_1-2w_1 = ks_{2i}$ is odd, and $4w_1 = kt_{2i}$ is even. This implies $s_{2i}$ is odd, and $t_{2i}$ is even. Using \eqref{eq:cm1}, we deduce that $t_{2i-1}$ is odd. This implies $2y_1 = t_{2i-1}+lt_{2i}$ is odd, a contradiction. So $\beta$ is indecomposable in $\mc O_K^+$.
\item Suppose $x_1,y_1\in\frac 12+\Z$. In this case, $\beta_1\in\mc O_K$ implies $z_1,w_1\in\frac 14+\frac 12\Z$, and $z_1-w_1\in\frac 12+\Z$. Using \eqref{eq:pi3z} and that $\gcd(p,q) \equiv 1\pmod{4}$, we deduce that both $s_{2i},t_{2i}$ are odd. Using \eqref{eq:cm1}, we deduce that exactly one of $s_{2i-1},t_{2i-1}$ is even. But then $2x_1=s_{2i-1}+ls_{2i}$ and $2y_1=t_{2i-1}+lt_{2i}$ cannot both be odd, a contradiction. So $\beta$ is indecomposable in $\mc O_K^+$.
\ee

Finally, suppose the biquadratic field $K$ is of type (4). Then we have $\gcd(p,q) \equiv 3\pmod{4}$. Again, $x_1,y_1$ cannot both be integers, and there are three subcases, which can be treated completely analogously as above.
\end{proof}

\begin{prp}\label{prp:qi} 
\ber
\item Suppose $K$ is a real biquadratic field of type $(1)$ or $(2)$. Then the indecomposable elements in $\mc O_{K_q}^+$ remain indecomposable in $\mc O_K^+$.
\item Suppose $K$ is a real biquadratic field of type $(3)$ or $(4)$, and $q<r$. Then the indecomposable elements in $\mc O_{K_q}^+$ remain indecomposable in $\mc O_K^+$.
\ee
\end{prp}
\begin{proof}
The first statement is contained in \cite[Theorem 2.1]{CLSTZ2019}. For biquadratic fields $K$ of type $(3)$ and $(4)$, the numbers $p,q,r$ are all interchangeable. By swapping $p$ and $q$, the second statement follows from \Cref{prp:pi34}.
\end{proof}

\begin{proof}[Proof of \Cref{thm:pid}]
\Cref{prp:pi12,prp:pi34} show that indecomposable elements in $\mc O_{K_p}^+$ remain indecomposable in $\mc O_K^+$, and \Cref{prp:qi} shows that indecomposable elements in $\mc O_{K_q}^+$ remain indecomposable in $\mc O_K^+$.
\end{proof}

For cases not described by \Cref{thm:pid}, indecomposable elements in a quadratic subfield need not stay indecomposable in $\mc O_K^+$. We give several examples.
\begin{expl}
Let $p=14$, $q=91$, $r=26$. Then $K = \Q(\sqrt{p},\sqrt{q})$ is of type (1). Then $26+5\sqrt{26} \in \mc O_{K_r}^+$ is indecomposable in $\mc O_{K_r}^+$, but admits a decomposition
\ba
26+5\sqrt{26} = \rb{13+\frac 52\sqrt{14} + \sqrt{91} + \frac 52\sqrt{26}} + \rb{13-\frac 52\sqrt{14} - \sqrt{91} + \frac 52\sqrt{26}}
\ea
in $\mc O_K^+$.
\end{expl}
\begin{expl}
Let $p=5$, $q=13$, $r=65$. Then $K = \Q(\sqrt{p},\sqrt{q})$ is of type (3). Then $\frac{25+3\sqrt{65}}2 \in \mc O_{K_r}^+$ is indecomposable $\mc O_{K_r}^+$, but admits a decomposition
\ba
\frac{25+3\sqrt{65}}2 = \frac{25+5\sqrt{5}+3\sqrt{13}+3\sqrt{65}}{4} + \frac{25-5\sqrt{5}-3\sqrt{13}+3\sqrt{65}}{4}
\ea
in $\mc O_K^+$.
\end{expl}

\section{Indecomposable elements in some families of real biquadratic fields}\label{sect:if} 

Here we restrict our attention to some specific one-parameter families of real biquadratic fields $K$, and give a complete characterisation of the indecomposable elements in $\mc O_K^+$. 

First we consider the one-parameter family of biquadratic fields described in \Cref{thm:f1}. Let $n\ge 6$ be an integer such that
\ba
p = (2n-1)(2n+1), \quad q = (2n-1)(2n+3), \quad r = \frac{pq}{\gcd(p,q)^2} = (2n+1)(2n+3),
\ea
are squarefree integers. Then $K = \Q(\sqrt{p},\sqrt{q})$ is a biquadratic field of type (2). We have continued fractions
\ba
\sqrt{p} = [2n-1;\ol{1,4n-2}], \quad \frac{-1+\sqrt{q}}2 = [n-1;\ol{1,2n-1}], \quad \sqrt{r} = [2n+1;\ol{1,4n+2}].
\ea
Using results from \Cref{sect:qi}, it follows that the fundamental units of the quadratic subfields $K_p, K_q, K_r$ are given respectively by
\ba
\varepsilon_p = 2n+\sqrt{p}, \quad \varepsilon_q = \frac{2n+1+\sqrt{q}}2, \quad \varepsilon_r = (2n+2)+\sqrt{r}.
\ea
Note that all the fundamental units above are totally positive, and by \eqref{eq:qil} we have $\iota(K_p) = \iota(K_q) = \iota(K_r)=1$, so the totally positive units are the only indecomposable elements in $\mc O_{K_p}^+, \mc O_{K_q}^+$, and $\mc O_{K_r}^+$. 

Now we find a system of generators for the group of totally positive units $\mc O_K^{\times,+}$. We shall use some results from \cite{Kubota1956}. Let $F$ be a real quadratic field, and $\varepsilon \in \mc O_F^{\times,+}$ be a totally positive unit. Then there exist a unique, squarefree rational integer $\delta = \delta_F(\varepsilon)$ such that $\delta \varepsilon \in F^2$. In fact, we may take $\delta$ to be the squarefree part of $\Tr_{F/\Q}(\varepsilon+1)$. Observing that $\varepsilon\Tr_{F/\Q}(\varepsilon+1) = (\varepsilon+1)^2$, it follows that $\sqrt{\delta \varepsilon} \in \mc O_F^+$. Now let $K$ be a real biquadratic field, with quadratic subfields $K_1,K_2,K_3$, and with respective fundamental units $\varepsilon_1,\varepsilon_2,\varepsilon_3$. Suppose $\varepsilon_1,\varepsilon_2,\varepsilon_3$ are all totally positive. Then \cite[Hilfssatz~11]{Kubota1956} says that for $ i_1,i_2,i_3\in\cb{0,1}$, we have
\ba
\varepsilon_1^{i_1}\varepsilon_2^{i_2}\varepsilon_3^{i_3} \in K^2 \quad \text{ if and only if } \quad \delta_{K_1}(\varepsilon_1)^{i_1} \delta_{K_2}(\varepsilon_2)^{i_2}\delta_{K_3}(\varepsilon_3)^{i_3} \in K^2.
\ea

For our family of biquadratic fields, we have
\ba
\delta_p = \delta_{K_p}(\varepsilon_p) = 4n+2, \quad \delta_q = \delta_{K_q}(\varepsilon_q) = 2n+3, \quad \delta_r = \delta_{K_r}(\varepsilon_r) = 4n+6.
\ea
Using the assumption that $p,q,r$ are squarefree, it is straightforward to verify, using the criterion above, that a system of fundamental units for $\mc O_K^\times$ is given by $\varepsilon_p, \varepsilon_q, \sqrt{\varepsilon_p\varepsilon_r}$. On the other hand, we note that $\sqrt{\delta_p\varepsilon_p}\sqrt{\delta_r\varepsilon_r} = \sqrt{\delta_p\delta_r}\sqrt{\varepsilon_p\varepsilon_r} \in \mc O_K^+$ is totally positive, while $\delta_p\delta_r = 4(2n+1)(2n+3)$ is not a rational square, so $\sqrt{\varepsilon_p\varepsilon_r}$ is not totally positive. We thus conclude that the group $\mc O_K^{\times,+}$ of totally positive units is generated by $\varepsilon_p,\varepsilon_q,\varepsilon_r$. We have $\mc O_K^\times/\mc O_K^{\times,+} = \cb{\pm 1, \pm\sqrt{\varepsilon_p\varepsilon_r}}$ and in particular $[\mc O_K^\times : \mc O_K^{\times,+}] = 4$. We will make use of this fact in \Cref{sect:fmr}. 

Now we compute a complete set of indecomposable elements of $\mc O_K^+$ modulo totally positive units, following the strategy outlined in \cite{KT2023}. We consider the embedding 
\ba
\sigma : K\to\R^4, \quad \alpha \mapsto (\sigma_1(\alpha), \sigma_2(\alpha), \sigma_3(\alpha), \sigma_4(\alpha))
\ea
of $K$ into the Minkowski space. The images of the totally positive elements in $K$ then lie inside the positive octant $\R^{4,+} = \cbm{(x_1,\ldots,x_4)}{x_i > 0 \text{ for } i=1,\ldots,4}$. We further consider a fundamental domain $\mc F$ for the action of multiplication by (the images of) totally positive units $\varepsilon \in \mc O_K^{\times,+}$ on $\R^4$. By \cite[Theorem 1]{DF2014}, the fundamental domain $\mc F$ is covered by the simplicial cones
\ba
C_{xyz} := \R_{\ge 0} \pb{1,\varepsilon_x, \varepsilon_x\varepsilon_y, \varepsilon_x\varepsilon_y\varepsilon_z},
\ea
where $\cb{x,y,z}$ is a permutation of $\cb{p,q,r}$. To verify that there is no overlap between the interiors of the cones, it suffices to check that the expression
\[
\sgn(\sigma) \cdot \sgn\rb{\det\rb{1,\varepsilon_{\sigma(p)},\varepsilon_{\sigma(p)}\varepsilon_{\sigma(q)},\varepsilon_{\sigma(p)}\varepsilon_{\sigma(q)}\varepsilon_{\sigma(r)}}} 
\]
has the same sign for all permutations $\sigma$ of $\cb{p,q,r}$ (see \cite{Colmez1988,DF2014}). Given the explicit characterisations of the units $\varepsilon_p,\varepsilon_q,\varepsilon_r$ above, this is straightforward computation.

To ease the computations below, we translate the cone $C_{xyz}$ by $\varepsilon_x^{-1}$, and consider instead the cones
\ba
C'_{xyz} := \R_{\ge 0} \pb{\varepsilon_x^{-1}, 1, \varepsilon_y, \varepsilon_y\varepsilon_z}.
\ea
Note that all the indecomposable elements lying in the cone $C'_{xyz}$ actually lie in the parallelepiped
\ba
P_{xyz} := [0,1]\pb{\varepsilon_x^{-1}, 1, \varepsilon_y, \varepsilon_y\varepsilon_z}.
\ea
Moreover, except for the units $\varepsilon_x^{-1}, 1, \varepsilon_y, \varepsilon_y\varepsilon_z$, all indecomposable elements lying in the parallelepiped $P_{xyz}$ has all the coordinates (with respect to the basis $\varepsilon_x^{-1}, 1, \varepsilon_y, \varepsilon_y\varepsilon_z$) strictly less than $1$. 

For the computations below, we also need explicit expressions for the following units:
\ba
\varepsilon_p^{-1} &= 2n-\sqrt{p}, \quad \varepsilon_q^{-1} = \rb{n+\frac 12} - \frac 12\sqrt{q}, \quad \varepsilon_r^{-1} = (2n+2)-\sqrt{r},\\
\varepsilon_p\varepsilon_q &= (2n^2+n) + \rb{n+\frac 12}\sqrt{p} + n\sqrt{q} + \rb{n-\frac 12}\sqrt{r},\\
\varepsilon_p\varepsilon_r &= (4n^2+4n) + (2n+2)\sqrt{p} + (2n+1)\sqrt{q} + 2n\sqrt{r},\\
\varepsilon_q\varepsilon_r &= (2n^2+3n+1) + \rb{n+\frac 32}\sqrt{p} + (n+1)\sqrt{q} + \rb{n+\frac 12}\sqrt{r}.
\ea
Note that the biquadratic field $K$ is of type (2) as per the classification in \Cref{sect:bf}. Thus we have
\begin{equation}\label{eq:fbib} 
x+y\sqrt{p}+z\sqrt{q}+w\sqrt{r} \in \mc O_K \iff x,y,z,w\in\frac 12\Z, \quad x+z, y+w\in\Z.
\end{equation}

For the computations, it is also convenient to have good rational approximations of $\sqrt{p},\sqrt{q},\sqrt{r}$, coming from the continued fractions given above:
\begin{equation}\label{eq:re} 
\textstyle \frac{8n^2-2n-1}{4n-1}<\sqrt{p}<2n, \quad \frac{2n^2+n-1}{n}<\sqrt{q}<\frac{4n^2+4n-1}{2n+1}, \quad  \frac{8n^2+14n+5}{4n+3}<\sqrt{r}<2n+2.
\end{equation}

\subsection{The parallelepiped $P_{pqr}$}

We first consider the parallelepiped $P_{pqr}$ spanned by $\varepsilon_p^{-1},1,\varepsilon_q,\varepsilon_q\varepsilon_r$. Let 
\ba
\alpha = a\varepsilon_p^{-1} + b + c\varepsilon_q + d\varepsilon_q\varepsilon_r  = x+y\sqrt{p}+z\sqrt{q}+w\sqrt{r} \in \mc O_K, \quad 0 \le a,b,c,d < 1
\ea
be an integral point in the semi-open parallelepiped. Since $\varepsilon_q\varepsilon_r$ is the only basis element of the parallelepiped with a nonzero coefficient in $\sqrt{r}$, we use \eqref{eq:fbib} and the condition $0\le d < 1$ to conclude that $d=\frac{t}{2n+1}$ for some $t\in\cb{0,\ldots,2n}$. Also, \eqref{eq:fbib} implies $a-d, c+d, b-a \in\Z$. The condition $0\le a,b,c < 1$ then gives
\ba
(a,b,c,d) = \begin{cases} (0,0,0,0), \text{ or }\\ (\frac{t}{2n+1},\frac{t}{2n+1},1-\frac{t}{2n+1},\frac{t}{2n+1}) \quad (1\le t \le 2n).\end{cases}
\ea 
Define
\begin{equation}\label{eq:md} 
\mu := \frac{1}{2n+1} \rb{\varepsilon_p^{-1}+1-\varepsilon_q+\varepsilon_q\varepsilon_r} = \rb{n+\frac 32} + \frac 12\sqrt{p}+\frac 12\sqrt{q}+\frac 12\sqrt{r}.
\end{equation}
Then the integral points $\alpha$ lying in this semi-open parallelepiped are given by
\ba
\alpha = 0, \text{ or } \alpha = \varepsilon_q + t\mu \quad (1\le t \le 2n).
\ea
Using \eqref{eq:re}, we verify that $\mu$ is totally positive. It follows that none of the totally positive elements lying in this semi-open parallelepiped are indecomposable.

\subsection{The parallelepiped $P_{prq}$}

Now we consider the parallelepiped $P_{prq}$ spanned by $\varepsilon_p^{-1},1,\varepsilon_r,\varepsilon_r\varepsilon_q$. Let
\ba
\alpha = a\varepsilon_p^{-1} + b + c\varepsilon_r + d\varepsilon_r\varepsilon_q  = x+y\sqrt{p}+z\sqrt{q}+w\sqrt{r} \in \mc O_K, \quad 0 \le a,b,c,d < 1
\ea
be an integral point in the semi-open parallelepiped. Since $\varepsilon_r\varepsilon_q$ is the only basis element of the parallelepiped with a nonzero coefficient in $\sqrt{q}$, we use \eqref{eq:fbib} and the condition $0\le d < 1$ to conclude that $d = \frac t{2n+2}$ for some $t\in\cb{0,\ldots,2n+1}$. Also, \eqref{eq:fbib} implies $2a-d, 2c-d, a-c, 2a-b\in\Z$. The condition $0\le a,b,c < 1$ then gives
\ba
(a,b,c,d) = \begin{cases}(\frac{t}{4n+4},\frac{t}{2n+2},\frac{t}{4n+4},\frac{t}{2n+2}) & (0\le t \le 2n+1),\\ (\frac 12+\frac{t}{4n+4},\frac{t}{2n+2},\frac 12+\frac{t}{4n+4},\frac{t}{2n+2}) & (0\le t \le 2n+1).\end{cases}
\ea
It follows that the integral points $\alpha$ lying in this semi-open parallelepiped are given by
\ba
\alpha = t\mu \quad (0\le t \le 2n+1), \text{ or } \alpha=\frac{\varepsilon_p^{-1}+\varepsilon_r}2 + t\mu \quad (0\le t \le 2n+1),
\ea
where $\mu$ is as in \eqref{eq:md}. Since $\mu$ is totally positive, we see that only $\mu$ and $\frac{\varepsilon_p^{-1}+\varepsilon_r}2$ can be indecomposable. We show that this is indeed the case.
\begin{prp}\label{prp:prqi} 
The elements $\mu, \frac{\varepsilon_p^{-1}+\varepsilon_r}2 \in \mc O_K^+$ are indecomposable.
\end{prp}
\begin{proof}
We prove $\mu$ is indecomposable; the proof for $\frac{\varepsilon_p^{-1}+\varepsilon_r}2$ is similar. Suppose we have a decomposition
\ba
\mu = \beta_1+\beta_2, \quad \beta_i = x_i+y_i\sqrt{p}+z_i\sqrt{q}+w_i\sqrt{r} \in \mc O_K^+.
\ea
Since $w_1+w_2 = \frac 12$ and $w_1,w_2\in\frac 12\Z$, we may assume $w_1 \ge \frac 12$. Since $\beta_1$ is totally positive, we have
\ba
2x_1-2w_1\sqrt{r} > 0.
\ea
As $x_1< n+\frac 32$, this forces $w_1 = \frac 12$, and $x_1 = n+1$. But then we have $w_2 = 0$, and $x_2 = \frac 12$. The integrality condition \eqref{eq:fbib} implies $z_2 \in \frac 12 + \Z$. Meanwhile, $\beta_2$ being totally positive implies
\ba
\textstyle x_2 - \vb{z_2}\sqrt{q} = \frac 12 - \vb{z_2}\sqrt{q} > 0,
\ea
which is impossible, because $\vb{z_2}\ge \frac 12$, and $\sqrt{q} > 1$.
\end{proof}

\subsection{The parallelepiped $P_{qpr}$}

Now we consider the parallelepiped $P_{qpr}$ spanned by $\varepsilon_q^{-1},1,\varepsilon_p,\varepsilon_p\varepsilon_r$. Let
\ba
\alpha = a\varepsilon_q^{-1} + b + c\varepsilon_p + d\varepsilon_p\varepsilon_r  = x+y\sqrt{p}+z\sqrt{q}+w\sqrt{r} \in \mc O_K, \quad 0 \le a,b,c,d < 1
\ea
be an integral point in the semi-open parallelepiped. Since $\varepsilon_p\varepsilon_r$ is the only basis element of the parallelepiped with a nonzero coefficient in $\sqrt{r}$, we use \eqref{eq:fbib} and the condition $0\le d < 1$ to conclude that $d = \frac{t}{4n}$ for some $t\in\cb{0,\ldots,4n-1}$. Also, \eqref{eq:fbib} implies $a-2d,c+2d, a+2b\in\Z$. Then condition $0\le a, b, c < 1$ then gives 
\ba
(a,b,c,d) = \begin{cases} (0,0,0,0),\\ (\frac{t}{2n}, 1-\frac{t}{4n},1-\frac{t}{2n},\frac{t}{4n}) & (1 \le t \le 2n-1),\\ (0,\frac 12,0,\frac 12),\\ (\frac{t}{2n}, \frac 12-\frac{t}{4n},1-\frac{t}{2n},\frac 12+\frac{t}{4n}) & (1 \le t \le 2n-1).\end{cases}
\ea
It follows that the integral points $\alpha$ lying in this semi-open parallelepiped are given by
\ba
\alpha = \begin{cases} 0,\\ 1+\varepsilon_p + t(\mu-1) & (1\le t \le 2n-1),\\ \frac{1+\varepsilon_p\varepsilon_r}2,\\ \frac{1+\varepsilon_p\varepsilon_r}2 + \varepsilon_p + t(\mu-1) & (1\le t \le 2n-1),\end{cases}
\ea
where $\mu$ is as in \eqref{eq:md}. Note that $\mu-1$ is \emph{not} totally positive. We have explicit decompositions for the following elements:
\ba
1+\varepsilon_p + (\mu-1) &= 1 + \rb{\varepsilon_p+\mu-1},\\
1+\varepsilon_p + 2(\mu-1) &= \mu + \rb{\varepsilon_p+\mu-1},\\
1+\varepsilon_p + (2n-1)(\mu-1) &= \varepsilon_p\varepsilon_q + \varepsilon_q^{-1},\\
\frac{1+\varepsilon_p\varepsilon_r}2 + \varepsilon_p + (\mu-1) &= \frac{1+\varepsilon_p\varepsilon_r}2 + \rb{\varepsilon_p+\mu-1},\\
\frac{1+\varepsilon_p\varepsilon_r}2 + \varepsilon_p + 2(\mu-1) &= \rb{\frac{1+\varepsilon_p\varepsilon_r}2+\mu-1} + \rb{\varepsilon_p+\mu-1},\\
\frac{1+\varepsilon_p\varepsilon_r}2 + \varepsilon_p + (2n-1)(\mu-1) &= \rb{\varepsilon_p\varepsilon_r-\mu-1} + \varepsilon_q^{-1}.
\ea
Meanwhile, we have $\frac{1+\varepsilon_p\varepsilon_r}2 = \varepsilon_p \big(\frac{\varepsilon_p^{-1}+\varepsilon_r}2\big)$, so $\frac{1+\varepsilon_p\varepsilon_r}2$ is equivalent to the indecomposable element $\frac{\varepsilon_p^{-1}+\varepsilon_r}2$ found in \Cref{prp:prqi}. The following proposition says the remaining elements are indecomposable.
\begin{prp}\label{prp:qpri} 
For $3\le t \le 2n-2$, the elements $1+\varepsilon_p + t(\mu-1)$ and $\frac{1+\varepsilon_p\varepsilon_r}2 + \varepsilon_p + t(\mu-1)$ are indecomposable.
\end{prp}
\begin{proof}
We prove that $1+\varepsilon_p + t(\mu-1)$ is indecomposable for $3\le t \le 2n-2$; the argument for $\frac{1+\varepsilon_p\varepsilon_r}2 + \varepsilon_p + t(\mu-1)$ is similar. Throughout the proof, we keep in mind of the assumption that $n\ge 6$ and $3\le t \le 2n-2$. Write 
\ba
\gamma_t := 1+\varepsilon_p + t(\mu-1) = (t+2)\rb{n+\frac 12} + \rb{\frac t2+1}\sqrt{p} + \frac t2\sqrt{q} + \frac t2\sqrt{r},
\ea
and suppose we have a decomposition 
\ba
\gamma_t = \beta_1+\beta_2, \quad \beta_i = x_i+y_i\sqrt{p}+z_i\sqrt{q}+w_i\sqrt{r} \in \mc O_K^+.
\ea
Total positivity of $\beta_1, \beta_2$ implies
\begin{equation}\label{eq:dtp} 
\begin{aligned}
x_i - \vb{y_i}\sqrt{p} &>0, \quad & \quad x_i - \vb{z_i}\sqrt{q} &> 0, \quad & \quad x_i-\vb{w_i}\sqrt{r} &>0.
\end{aligned}
\end{equation}
Since $x_i \le (t+2)(n+\frac 12) - \frac 12$, this implies 
\begin{equation}\label{eq:yzwm} 
\vb{y_i}\le \frac t2+1, \quad \vb{z_i}\le \frac t2+1, \quad \vb{w_i}\le \frac{t+1}2.
\end{equation}
It follows that
\ba
\textstyle 0 \le y_i \le \frac t2+1, \quad -1\le z_i \le \frac t2+1, \quad -\frac 12\le w_i \le \frac{t+1}2.
\ea
Total positivity of $\beta_1, \beta_2$ also implies
\begin{equation}\label{eq:qpr_s24} 
x_1-y_1\sqrt{p}-\vb{z_1\sqrt{q}-w_1\sqrt{r}} > 0, \quad x_2-y_2\sqrt{p}-\vb{z_2\sqrt{q}-w_2\sqrt{r}} > 0,
\end{equation}
and
\begin{equation}\label{eq:qpr_s13} 
x_1+y_1\sqrt{p}-\vb{z_1\sqrt{q}+w_1\sqrt{r}} > 0, \quad x_2+y_2\sqrt{p}-\vb{z_2\sqrt{q}+w_2\sqrt{r}} > 0. 
\end{equation}

Now we claim that $\vb{z_i-w_i}\le \frac 12$. To see this, we suppose to the contrary that $\vb{z_i-w_i}\ge 1$; by swapping indices, we may assume $w_1-z_1\ge 1$. Then, using \eqref{eq:qpr_s24} we have
\ba
(t+2)\rb{n+\frac 12} - \rb{\frac t2+1}\sqrt{p} &= \rb{x_1+x_2}-\rb{y_1+y_2}\sqrt{p} > \vb{z_1\sqrt{q}-w_1\sqrt{r}}+\vb{z_2\sqrt{q}-w_2\sqrt{r}}.
\ea
Meanwhile, using that $w_1 \ge z_1+1 \ge 0$, we obtain the inequality
\ba
\vb{z_1\sqrt{q}-w_1\sqrt{r}} = \vb{\rb{z_1-w_1}\sqrt{q} - w_1\rb{\sqrt{r}-\sqrt{q}}} \ge \sqrt{q}.
\ea
This implies 
\ba
2n & \ge (t+2)\rb{n+\frac 12} - \rb{\frac t2+1}(2n-1) > (t+2)\rb{n+\frac 12} - \rb{\frac t2+1}\sqrt{p}\\
&> \vb{z_1\sqrt{q}-w_1\sqrt{r}}+\vb{z_2\sqrt{q}-w_2\sqrt{r}} > \sqrt{q},
\ea
a contradiction. So the claim is established, and we are left with two possibilities, namely $z_i=w_i$, and $\vb{z_i-w_i} = \frac 12$. 

First we consider the case $z_i=w_i$. We note that our assumptions above imply $x_i \ge 2ny_i$. To see this, we suppose to the contrary that $x_i \le 2n y_i-\frac 12$. Then \eqref{eq:dtp} gives $(2n-\sqrt{p})y_i > \frac 12$. Using \eqref{eq:re}, this gives $\frac{y_i}{4n-1} > \frac 12$, hence $y_i \ge 2n > \frac t2+1$, which contradicts \eqref{eq:yzwm}. So we may write $x_i+y_i\sqrt{p} = x'_i + y_i(2n+\sqrt{p})$, with $x'_i \ge 0$.

Next we claim that $x'_i \ge \vb{z_i}$. Suppose this is not the case, and we have $x'_i \le \vb{z_i}-1$ (note that we have $x'_i-z_i\in\Z$). Then \eqref{eq:qpr_s24} says
\ba
y_i \rb{2n-\sqrt{p}} > \vb{z_i}\rb{\sqrt{r}-\sqrt{q}} - x'_i \ge \vb{z_i}\rb{\sqrt{r}-\sqrt{q}-1}+1 \ge 1.
\ea
Using \eqref{eq:re}, we get $y_i > (2n-\sqrt{p})^{-1} > 4n-1$, which contradicts \eqref{eq:yzwm}. So we have $x'_i \ge \vb{z_i}$ as claimed. Since $x'_1+x'_2 = \frac t2+1 = z_1+z_2+1$, and $x'_i-z_i \in\Z$, it follows that $x'_i = \vb{z_i}$ or $\vb{z_i}+1$. By swapping indices, we may actually assume that $(x'_1,x'_2) = (\vb{z_1},\vb{z_2})$ or $(\vb{z_1}+1,\vb{z_2})$. 
\ber
\item Suppose $(x'_1,x'_2) = (\vb{z_1},\vb{z_2})$. Then we use \eqref{eq:qpr_s24} and obtain the inequality
\ba
y_i \rb{2n-\sqrt{p}} > \vb{z_i}\rb{\sqrt{r}-\sqrt{q}} - x'_i = \vb{z_i}\rb{\sqrt{r}-\sqrt{q}-1}, \quad i\in\cb{1,2}.
\ea
Using \eqref{eq:re}, this yields
\ba
y_i > \vb{z_i}\frac{\sqrt{r}-\sqrt{q}-1}{2n-\sqrt{p}} > \vb{z_i}\rb{3-\frac{16n+14}{8n^2+10n+3}}.
\ea
For $n\ge 6$, we verify that this inequality implies
\begin{equation}\label{eq:qpr_yi} 
y_i \ge \begin{cases} 3\vb{z_i} & \text{ if } \vb{z_i}\le 3, \\ 3\vb{z_i}-1 & \text{ otherwise.}\end{cases}
\end{equation}
\item Now suppose $(x'_1,x'_2) = (\vb{z_1}+1,\vb{z_2})$. In this case, the bound \eqref{eq:qpr_yi} still holds for $i=2$. On the other hand, we use \eqref{eq:qpr_s13} and obtain the inequality
\ba
y_1 \rb{2n+\sqrt{p}} > \vb{z_1}\rb{\sqrt{r}+\sqrt{q}} - x'_1 = \vb{z_1}\rb{\sqrt{r}+\sqrt{q}-1} - 1.
\ea
Using \eqref{eq:re}, this yields
\ba
y_1 > \vb{z_1}\frac{\sqrt{r}+\sqrt{q}-1}{2n+\sqrt{p}} - \frac{1}{2n+\sqrt{p}} > \vb{z_1} \rb{1+\frac{8n^2+n-3}{16n^3+12n^2}} - \frac{4n-1}{16n^2-4n-1}.
\ea
This implies
\ba
y_1 \ge \begin{cases} \vb{z_1} & \text{ if } \vb{z_1}\le \frac 12,\\ \vb{z_1}+1 & \text{ otherwise.}\end{cases}
\ea
\ee
After excluding the trivial decomposition $(x'_i,y_i,z_i,w_i) = (0,0,0,0)$, we conclude in each of the cases that when $\vb{z_1}+\vb{z_2} \ge \frac 32$, we always have $y_1+y_2 \ge \vb{z_1}+\vb{z_2}+2 >  \frac t2+1$, a contradiction. So no decomposition of $\gamma_t$ with $z_i=w_i$ exists.

Now we consider the case $\vb{z_i-w_i} = \frac 12$. We may actually assume $w_1 = z_1 -\frac 12$. Then \eqref{eq:yzwm} says $0\le z_1 \le \frac t2+1$. For $3\le t \le 2n-2$, we check using \eqref{eq:re} that
\ba
z_1\sqrt{q} - \rb{z_1-\frac 12}\sqrt{r} > 0, \quad \rb{\frac t2-z_1}\sqrt{q} - \rb{\frac{t+1}2-z_1}\sqrt{r}<0.
\ea
It follows that
\ba
\vb{z_1\sqrt{q}-w_1\sqrt{r}} + \vb{z_2\sqrt{q}-w_2\sqrt{r}} &= \vb{z_1\sqrt{q}-\rb{z_1-\frac 12}\sqrt{r}} + \vb{\rb{\frac t2-z_1}\sqrt{q} - \rb{\frac{t+1}2-z_1}\sqrt{r}}\\
&= \sqrt{r}+\rb{\frac t2-2z_1}\rb{\sqrt{r}-\sqrt{q}} \ge \rb{\frac t2+2}\sqrt{q} - \rb{\frac t2+1}\sqrt{r}.
\ea
Using \eqref{eq:qpr_s24}, we get
\ba
\rb{\frac t2+1} \rb{2n+1-\sqrt{p}} > \rb{\frac t2+2}\sqrt{q} - \rb{\frac t2+1}\sqrt{r},
\ea
which can be checked using \eqref{eq:re} that it is not solvable for $n\ge 6$ and $3\le t \le 2n-3$. Therefore, for $3\le t \le 2n-3$, no decomposition of $\gamma_t$ with $\vb{z_i-w_i} = \frac 12$ exists. Now let $t=2n-2$. Then \eqref{eq:qpr_s24} implies
\ba
\rb{\frac t2+1}\rb{2n+1-\sqrt{p}} > \sqrt{r}+\rb{\frac t2-2z_1}\rb{\sqrt{r}-\sqrt{q}},
\ea
which implies $z_1 = \frac t2 + 1$ using \eqref{eq:re}. But then \eqref{eq:qpr_s13} and \eqref{eq:re} implies 
\ba
4n^2+n > n \rb{2n+1+\sqrt{p}} > (n+1)\sqrt{q} + n\sqrt{r} > \frac{16n^4+32n^3+14n^2-4n-3}{4n^2+3n}
\ea
which does not hold for $n\ge 6$. So $\gamma_{2n-2}$ is also indecomposable. This finishes the proof of the proposition.
\end{proof}

\subsection{The parallelepiped $P_{qrp}$} Now we consider the parallelepiped $P_{qrp}$ spanned by $\varepsilon_q^{-1},1,\varepsilon_r,\varepsilon_r\varepsilon_p$. Let
\ba
\alpha = a\varepsilon_q^{-1} + b + c\varepsilon_r + d\varepsilon_r\varepsilon_p = x+y\sqrt{p}+z\sqrt{q}+w\sqrt{r} \in \mc O_K, \quad 0 \le a,b,c,d < 1
\ea
be an integral point in the semi-open parallelepiped. Since $\varepsilon_r\varepsilon_p$ is the only basis element of the parallelepiped with a nonzero coefficient in $\sqrt{p}$, we use \eqref{eq:fbib} and the condition $0\le d < 1$ to conclude that $d=\frac t{4n+4}$ for some $t\in\cb{0,\ldots,4n+3}$. Also,\eqref{eq:fbib} implies $a+2d, c-2d, a-2b\in\Z$. The condition $0\le a,b,c < 1$ then gives
\ba
(a,b,c,d) = \begin{cases} (0,0,0,0),\\ (1-\frac{t}{2n+2}, 1-\frac{t}{4n+4}, \frac{t}{2n+2}, \frac{t}{4n+4}) & (1\le t \le 2n+1),\\ (0,\frac 12,0,\frac 12),\\ (1-\frac{t}{2n+2}, \frac 12 - \frac{t}{4n+4}, \frac{t}{2n+2}, \frac 12+\frac{t}{4n+4}) & (1\le t \le 2n+1).\end{cases}
\ea
It follows that the integral points $\alpha$ lying in this semi-open parallelepiped are given by
\ba
\alpha = \begin{cases} 0,\\ 1+\varepsilon_q^{-1}+t(\mu-1) & (1\le t \le 2n+1),\\ \frac{1+\varepsilon_r\varepsilon_p}2,\\ \frac{1+\varepsilon_r\varepsilon_p}2 + \varepsilon_q^{-1} + t(\mu-1) & (1\le t \le 2n+1), \end{cases}
\ea
where $\mu$ is as in \eqref{eq:md}. The element $\frac{1+\varepsilon_r\varepsilon_p}2 = \varepsilon_p\big(\frac{\varepsilon_p^{-1}+\varepsilon_r}2\big)$ is equivalent to the indecomposable element $\frac{\varepsilon_p^{-1}+\varepsilon_r}2$ found in \Cref{prp:prqi}. For the other elements, we check using arguments analogous to \Cref{prp:qpri} that $1+\varepsilon_q^{-1}+t(\mu-1)$ and $\frac{1+\varepsilon_r\varepsilon_p}2 + \varepsilon_q^{-1} + t(\mu-1)$ for $4\le t \le 2n-1$ are the indecomposable elements.

\subsection{The parallelepiped $P_{rpq}$} Now we consider the parallelepiped $P_{rpq}$ spanned by $\varepsilon_r^{-1},1,\varepsilon_p,\varepsilon_p\varepsilon_q$. Let
\ba
\alpha = a\varepsilon_r^{-1} + b + c\varepsilon_p + d\varepsilon_p\varepsilon_q = x+y\sqrt{p}+z\sqrt{q}+w\sqrt{r} \in \mc O_K, \quad 0 \le a,b,c,d < 1
\ea
be an integral point in the semi-open parallelepiped. Since $\varepsilon_p\varepsilon_q$ is the only basis element of the parallelepiped with a nonzero coefficient in $\sqrt{q}$, we use \eqref{eq:fbib} and the condition $0\le d < 1$ to conclude that $d = \frac{t}{2n}$ for some $t \in \cb{0,\ldots, 2n-1}$. Also, \eqref{eq:fbib} implies $2a+d, 2c+d, a-c, 2a+b\in\Z$. The condition $0\le a,b,c < 1$ then gives
\ba
(a,b,c,d) = \begin{cases} (0,0,0,0),\\ (1-\frac{t}{4n}, \frac{t}{2n}, 1-\frac{t}{4n}, \frac{t}{2n}) &(1\le t \le 2n-1),\\ (\frac 12-\frac{t}{4n}, \frac{t}{2n}, \frac 12-\frac{t}{4n}, \frac{t}{2n}) & (0\le t \le 2n-1).\end{cases}
\ea
It follows that the integral points $\alpha$ lying in this semi-open parallelepiped are given by
\ba
\alpha = \begin{cases} 0,\\ \varepsilon_r^{-1}+\varepsilon_p+t(\mu-2) & (1\le t \le 2n-1),\\ \frac{\varepsilon_r^{-1}+\varepsilon_p}2 + t(\mu-2) & (0\le t \le 2n-1),\end{cases}
\ea
where $\mu$ is as in \eqref{eq:md}. Note that $\mu-2$ is \emph{not} totally positive. The element $\frac{\varepsilon_r^{-1}+\varepsilon_p}2 = \varepsilon_p\varepsilon_r^{-1} \big(\frac{\varepsilon_p^{-1}+\varepsilon_r}2\big)$ is equivalent to the indecomposable element $\frac{\varepsilon_p^{-1}+\varepsilon_r}2$ found in \Cref{prp:prqi}. For the other elements, we check using arguments analogous to \Cref{prp:qpri} that $\frac{\varepsilon_r^{-1}+\varepsilon_p}2 + t(\mu-2)$ for $2\le t \le 2n-1$ are the indecomposable elements.

\subsection{The parallelepiped $P_{rqp}$} Now we consider the parallelepiped $P_{rqp}$ spanned by $\varepsilon_r^{-1},1,\varepsilon_q,\varepsilon_q\varepsilon_p$. Let
\ba
\alpha = a\varepsilon_r^{-1} + b + c\varepsilon_q + d\varepsilon_q\varepsilon_p = x+y\sqrt{p}+z\sqrt{q}+w\sqrt{r} \in \mc O_K, \quad 0 \le a,b,c,d < 1
\ea
be an integral point in the semi-open parallelepiped. Since $\varepsilon_q\varepsilon_p$ is the only basis element of the parallelepiped with a nonzero coefficient in $\sqrt{p}$, we use \eqref{eq:fbib} and the condition $0\le d < 1$ to conclude that $d = \frac{t}{2n+1}$ for some $t\in \cb{0,\ldots, 2n}$. Also, \eqref{eq:fbib} implies $a+d,c-d,a+b\in\Z$. The condition $0\le a,b,c < 1$ then gives
\ba
(a,b,c,d) = \begin{cases} (0,0,0,0), \text{ or }\\ (1-\frac{t}{2n+1},\frac{t}{2n+1},\frac{t}{2n+1},\frac{t}{2n+1}) & (1\le t \le 2n).\end{cases}
\ea
It follows that the integral points $\alpha$ lying in this semi-open parallelepiped are given by
\ba
\alpha = 0, \text{ or } \alpha = \varepsilon_r^{-1} + t(\mu-2) \quad (1\le t \le 2n),
\ea
where $\mu$ is as in \eqref{eq:md}. For $1\le t \le 2n$, we check that $\varepsilon_r^{-1} + t(\mu-2) - \varepsilon_q \in \mc O_K^+$. It follows that none of the totally positive elements lying in this semi-open parallelepiped are indecomposable.

Combining the computations above, we obtain a complete set of indecomposable elements in $\mc O_K^+$ modulo totally positive units. The proof of \Cref{thm:f1} is thus finished.

\subsection{More families}\label{sect:mbf}

We apply analogous arguments and compute the indecomposable elements for some other families of biquadratic fields. 

Consider the family $K = \Q(\sqrt{p},\sqrt{q})$, with $p=(2n-1)(2n+1)$, $q=(4n-3)(4n+1)$, where $9\le n \in \N$ is chosen such that $p,q$ are coprime and squarefree. Then we have $r = pq = (8n^2-2n-3)(8n^2-2n-1)$. We have continued fractions
\ba
\sqrt{p} = [2n-1;\ol{1,4n-2}], \quad \frac{-1+\sqrt{q}}2 = [2n-2;\ol{1,4n-3}], \quad \sqrt{r} = [8n^2-2n-3;\ol{1,16n^2-4n-6}].
\ea
Using results from \Cref{sect:qi}, the fundamental units of the quadratic subfields $K_p, K_q, K_r$ are given respectively by
\ba
\varepsilon_p = 2n+\sqrt{p}, \quad \varepsilon_q = \frac{4n-1+\sqrt{q}}2, \quad \varepsilon_r = (8n^2-2n-2)+\sqrt{r}.
\ea
All the fundamental units above are totally positive, and by \eqref{eq:qil} we have $\iota(K_p) = \iota(K_q) = \iota(K_r)=1$, so the totally positive units are the only indecomposable elements in $\mc O_{K_p}^+, \mc O_{K_q}^+$, and $\mc O_{K_r}^+$. Using Kubota's criterion, we verify that $\mc O_K^{\times,+}$ is generated by $\varepsilon_p, \varepsilon_q, \varepsilon_r$. Via analogous computations we obtain the following theorem.

\begin{thm}\label{thm:f2} 
Let $n\ge 9$ be an integer such that $p = (2n-1)(2n+1)$, $q = (4n-3)(4n+1)$ are coprime, squarefree integers. Let $K = \Q(\sqrt{p},\sqrt{q})$, and $r=pq$. A complete set of indecomposable elements of $\mc O_K^+$ modulo totally positive units is given by
\ba
&1, \quad \frac{\varepsilon_p^{-1}+\varepsilon_r}2 + t\rb{\varepsilon_q-\varepsilon_p^{-1}} & (0\le t \le 2n-2),\\
&\rb{\varepsilon_p^{-1}-\varepsilon_q} + t\rb{\varepsilon_p\varepsilon_q-1} & (1\le t \le 2n-1).
\ea
In particular, we have $\iota(K) = 4n-1$.
\end{thm}

Next we consider the family $K = \Q(\sqrt{p},\sqrt{q})$, with $p=(2n-1)(2n+1)$, $q=(4n-1)(4n+3)$, where $2\le n \in\N$ is chosen such that $p,q$ are coprime and squarefree. Then we have $r = pq = (8n^2+2n-3)(8n^2+2n-1)$. We have continued fractions
\ba
\sqrt{p} = [2n-1;\ol{1,4n-2}], \quad \frac{-1+\sqrt{q}}2 = [2n-1;\ol{1,4n-1}], \quad \sqrt{r} = [8n^2+2n-3;\ol{1,16n^2+4n-6}].
\ea
Using results from \Cref{sect:qi}, the fundamental units of the quadratic subfields $K_p, K_q, K_r$ are given respectively by
\ba
\varepsilon_p = 2n+\sqrt{p}, \quad \varepsilon_q = \frac{4n+1+\sqrt{q}}2, \quad \varepsilon_r = (8n^2+2n-2)+\sqrt{r}.
\ea
Again, all the fundamental units above are totally positive, and by \eqref{eq:qil} we have $\iota(K_p) = \iota(K_q) = \iota(K_r)=1$, so the totally positive units are the only indecomposable elements in $\mc O_{K_p}^+, \mc O_{K_q}^+$, and $\mc O_{K_r}^+$. Using Kubota's criterion, we verify that $\mc O_K^{\times,+}$ is generated by $\varepsilon_p, \varepsilon_q, \varepsilon_r$. Via analogous computations we obtain the following theorem.

\begin{thm}\label{thm:f3} 
Let $n\ge 2$ be an integer such that $p=(2n-1)(2n+1)$, $q=(4n-1)(4n+3)$ are coprime, squarefree integers. Let $K = \Q(\sqrt{p},\sqrt{q})$, and $r=pq$. A complete set of indecomposable elements of $\mc O_K^+$ modulo totally positive units is given by
\ba
&1, \quad \frac{\varepsilon_p^{-1}+\varepsilon_r}2, \quad  \frac{\varepsilon_p+\varepsilon_r}2 + t\rb{\varepsilon_p-\varepsilon_q^{-1}} & (0\le t \le 2n-2),\\
&\rb{\varepsilon_q^{-1}-\varepsilon_p} + t\rb{\varepsilon_p\varepsilon_q-1} & (1\le t \le 2n-1).
\ea
In particular, we have $\iota(K) = 4n$.
\end{thm}

\begin{rmk}
The families in \Cref{thm:f2,thm:f3} are also infinite. This can be proved using sieve arguments (see the remark after \Cref{thm:f1}).
\end{rmk}

\subsection{Minimal rank of universal lattices}\label{sect:fmr}

Using \Cref{prp:ufi}, we can derive easy upper bounds on the minimal rank of universal lattices for real biquadratic fields. It is known that for a biquadratic field $K$ the Pythagoras number $s(\mc O_K)$ of $K$ is at most $7$ \cite[Corollary 3.3]{KY2021}. It follows from the proof of \Cref{thm:fi} that we have $R(K) \le R_{\operatorname{cls}}(K) \le 56 \iota(K)$.

For the one-parameter families above, we are able to further improve this bound, because we can explicitly compute the size of the quotient $[\mc O_K^{\times,+}:(\mc O_K^\times)^2]$, using Kubota's criterion \cite{Kubota1956}.

\begin{cor}
Assume the setup in \Cref{thm:f1}. Then we have $R(K) \le R_{\operatorname{cls}}(K) \le 28(10n-15)$.
\end{cor}
\begin{proof}
For this family, we have seen in the beginning of the section that $[\mc O_K^\times:\mc O_K^{\times,+}] = 4$. It follows that $[\mc O_K^{\times,+}: (\mc O_K^\times)^2] = 4$. Combining this result with \Cref{thm:f1} and \Cref{prp:ufi} yields the bound.
\end{proof}

For the families in \Cref{thm:f2,thm:f3}, similar computations show that $[\mc O_K^{\times,+}: (\mc O_K^\times)^2] = 4$ as well. So we have the following:

\begin{cor}
Assume the setup in \Cref{thm:f2}. Then we have $R(K) \le R_{\operatorname{cls}}(K) \le 28(4n-1)$.
\end{cor}
\begin{cor}
Assume the setup in \Cref{thm:f3}. Then we have $R(K) \le R_{\operatorname{cls}}(K) \le 112n$.
\end{cor}

\subsection{Some computational results}\label{sect:cr} 

Here we compute $\iota(K)$ for some real biquadratic fields $K = \Q(\sqrt{p},\sqrt{q})$ with small discriminant. We also provide the ratio $\frac{\log(\iota(K))}{\log(\Delta_K)}$ as reference. For the fields marked with an asterisk, there are ``extra'' totally positive units which are not products of totally positive units in the quadratic subfields. This partly accounts for their relatively low number of indecomposable elements compared to other fields of similar discriminants.
\ba
\scalebox{0.9}{$\begin{array}{c|c|c|c|c|c}
(p,q,r) & \iota(K_p) & \iota(K_q) & \iota(K_r) & \iota(K) & \frac{\log(\iota(K))}{\log(\Delta_K)}\\
\hline
(2,3,6)^* & 2 & 1 & 2 & 5 & 0.2532\\
(2,5,10)^* & 2 & 1 & 6 & 14 & 0.3577\\
(2,7,14)^* & 2 & 2 & 2 & 4 & 0.1722\\
(3,5,15) & 1 & 1 & 1 & 3 & 0.1342\\
(3,21,7) & 1 & 1 & 2 & 15 & 0.3056\\
(2,11,22)^* & 2 & 3 & 6 & 79 & 0.4879\\
(2,13,26)^* & 2 & 3 & 10 & 56 & 0.4334\\
(2,15,30) & 2 & 1 & 2 & 28 & 0.3480\\
(6,5,30) & 2 & 1 & 2 & 16 & 0.2896\\
(6,15,10)^* & 2 & 1 & 6 & 12 & 0.2595\\
(10,3,30)^* & 6 & 1 & 2 & 21 & 0.3180\\
(3,33,11) & 1 & 4 & 3 & 59 & 0.4175\\
(2,17,34)^* & 2 & 5 & 2 & 24 & 0.3235\\
(7,5,35) & 2 & 1 & 1 & 11 & 0.2426\\
(2,19,38)^* & 2 & 7 & 6 & 111 & 0.4687\\
(3,13,39)^* & 1 & 3 & 4 & 13 & 0.2540\\
(2,21,42)^* & 2 & 1 & 2 & 2 & 0.0676\\
(6,7,42) & 2 & 2 & 2 & 64 & 0.4058\\
(6,21,14) & 2 & 1 & 2 & 48 & 0.3778\\
(14,3,42) & 2 & 1 & 2 & 40 & 0.3600\\
(3,17,51)^* & 1 & 5 & 7 & 47 & 0.3620\\
(11,5,55)^* & 3 & 1 & 4 & 9 & 0.2037\\
(3,57,19) & 1 & 7 & 7 & 243 & 0.5059\\
(5,13,65)^* & 1 & 3 & 9 & 35 & 0.3197\\
(2,33,66) & 2 & 4 & 8 & 198 & 0.4742\\
\end{array}\quad
\begin{array}{c|c|c|c|c|c}
(p,q,r) & \iota(K_p) & \iota(K_q) & \iota(K_r) & \iota(K) & \frac{\log(\iota(K))}{\log(\Delta_K)}\\
\hline
(6,11,66) & 2 & 3 & 8 & 170 & 0.4605\\
(6,33,22)^* & 2 & 4 & 6 & 120 & 0.4293\\
(22,3,66) & 6 & 1 & 8 & 226 & 0.4861\\
(3,69,23)^* & 1 & 2 & 2 & 7 & 0.1731\\
(2,35,70)^* & 2 & 1 & 6 & 13 & 0.2276\\
(10,7,70) & 6 & 2 & 6 & 232 & 0.4833\\
(10,35,14) & 6 & 1 & 2 & 114 & 0.4203\\
(14,5,70) & 2 & 1 & 6 & 76 & 0.3843\\
(7,77,11)^* & 2 & 1 & 3 & 16 & 0.2419\\
(2,39,78) & 2 & 4 & 2 & 150 & 0.4362\\
(6,13,78)^* & 2 & 3 & 2 & 22 & 0.2691\\
(6,39,26) & 2 & 4 & 10 & 136 & 0.4277\\
(26,3,78) & 10 & 1 & 2 & 102 & 0.4027\\
(5,17,85)^* & 1 & 5 & 9 & 51 & 0.3373\\
(3,29,87) & 1 & 5 & 3 & 45 & 0.3252\\
(19,5,95)^* & 7 & 1 & 2 & 13 & 0.2159\\
(2,51,102)^* & 2 & 7 & 10 & 111 & 0.3917\\
(6,17,102)^* & 2 & 5 & 10 & 116 & 0.3954\\
(6,51,34) & 2 & 7 & 2 & 222 & 0.4494\\
(34,3,102) & 2 & 1 & 10 & 130 & 0.4049\\
(3,105,35)^* & 1 & 3 & 1 & 3 & 0.0909\\
(5,21,105) & 1 & 1 & 3 & 14 & 0.2185\\
(7,105,15) & 2 & 3 & 1 & 115 & 0.3928\\
(15,21,35) & 1 & 1 & 1 & 5 & 0.1332\\
(2,55,110) & 2 & 4 & 2 & 132 & 0.4011\\
\end{array}$}
\ea
The examples above suggest that biquadratic fields tend to have more indecomposable elements, if their quadratic subfields do. Based on this observation, we compile another list of biquadratic fields $K$, for which every quadratic subfield $F\sbe K$ satisfies $\iota(F) = 1$, and does not lie in the aforementioned families:
\ba
\begin{array}{c|c|c}
(p,q,r) & \iota(K) & \frac{\log(\iota(K))}{\log(\Delta_K)}\\
\hline
(3,5,15) & 3 & 0.1342\\
(35,357,255) & 38 & 0.2050\\
(35,285,399) & 32 & 0.1929\\
(143,1365,1155) & 98 & 0.2083\\
(143,1221,1443) & 92 & 0.2044\\
(195,8645,399) & 120 & 0.2073\\
(255,13685,483) & 150 & 0.2086\\
(195,1085,8463) & 150 & 0.2080\\
(323,3021,2703) & 158 & 0.2069\\
(323,3021,2703) & 152 & 0.2047\\
(399,1085,8835) & 179 & 0.2089\\
(195,957,20735) & 166 & 0.2057\\
\end{array} \quad 
\begin{array}{c|c|c}
(p,q,r) & \iota(K) & \frac{\log(\iota(K))}{\log(\Delta_K)}\\
\hline
(483,12765,1295) & 174 & 0.2017\\
(255,1221,34595) & 212 & 0.2070\\
(899,8277,7743) & 278 & 0.2043\\
(899,7917,8463) & 272 & 0.2031\\
(483,1677,89999) & 353 & 0.2111\\
(1023,24645,2915) & 256 & 0.1995\\
(1295,11877,11235) & 338 & 0.2033\\
(1295,11445,12099) & 332 & 0.2024\\
(1155,7917,20735) & 338 & 0.2026\\
(1443,3965,33855) & 401 & 0.2083\\
(1023,164021,1443) & 480 & 0.2130\\
(1023,5621,47523) & 450 & 0.2099\\
\end{array}
\ea

As a comparison, the biquadratic fields from \Cref{thm:f2,thm:f3} satisfy $\iota(K) \asymp \Delta_K^{1/8}$. This suggests that these fields are among those for which $\iota(K)$ are the smallest.

In view of the trends we found above, we expect \Cref{cnj:cnj} to be true. This however appears to be out of reach with current techniques, mainly because very little is known about the structure of indecomposable elements in biquadratic fields, and there is no known algorithm which finds these indecomposable elements efficiently.


\newcommand{\etalchar}[1]{$^{#1}$}

\end{document}